\documentclass[11pt]{article}

\usepackage[T1]{fontenc}
\usepackage[margin=1in]{geometry}
\usepackage{amsmath,amssymb,amsthm}
\usepackage{bbm}
\usepackage{graphicx}
\usepackage[numbers,sort&compress]{natbib}
\usepackage[colorlinks,citecolor=blue,urlcolor=blue]{hyperref}
\usepackage{cleveref}
\newtheorem{theorem}{Theorem}
\newtheorem{lemma}{Lemma}
\newtheorem{proposition}{Proposition}
\newtheorem{corollary}{Corollary}

\crefname{appendix}{Appendix}{Appendices}
\Crefname{appendix}{Appendix}{Appendices}
\crefformat{equation}{(#2#1#3)}
\crefmultiformat{equation}
  {(#2#1#3)}
  { and~(#2#1#3)}
  {, (#2#1#3)}
  {, and~(#2#1#3)}

\newcommand{\R}{\mathbb R}
\newcommand{\E}{\mathbb E}
\newcommand{\Pp}{\mathbb P}
\newcommand{\1}{\mathbbm 1}
\newcommand{\dd}{\,\mathrm d}
\newcommand{\abs}[1]{\lvert #1\rvert}
\newcommand{\norm}[1]{\lVert #1\rVert}
\newcommand{\Dspace}{\mathbb D}
\newcommand{\jumpgen}{\mathcal J}
\newcommand{\jumpgenI}{\mathcal J_I}
\newcommand{\SkMap}{\Gamma_{\mathrm{Sk}}}
\newcommand{\RegMap}{\mathcal R}
\newcommand{\law}{\operatorname{Law}}
\newcommand{\e}{\mathrm e}
\newcommand{\cA}{c_{\mathrm A}}
\newcommand{\Stable}{\mathcal S_\alpha}
\newcommand{\gen}{\mathcal L}

\title{Exponential Ergodicity and Invariant Distribution of a Degenerate Reflected SDE with Stable Noise}
\author{Guodong Pang\thanks{Department of Computational Applied Mathematics and Operations Research, Rice University. Email: \href{mailto:gdpang@rice.edu}{\texttt{gdpang@rice.edu}}.}
\and Wei You\thanks{Department of Industrial Engineering and Decision Analytics, The Hong Kong University of Science and Technology. Email: \href{mailto:weiyou@ust.hk}{\texttt{weiyou@ust.hk}}.}}
\date{}

\begin{document}
\maketitle

\begin{abstract}
We study a two-dimensional degenerate stochastic differential equation with oblique reflection and spectrally positive strictly stable noise of index in $(1,2)$, motivated by join-the-shortest-queue systems with heavy-tailed interarrival times.
A polynomial--exponential Lyapunov function and fixed-time minorization establish positive Harris recurrence and exponential ergodicity in weighted total variation.
We construct the invariant law from a stable path and prove its absolute continuity.
Regulator identities describe stationary dependence and recover the joint Laplace transform from the queue marginal.
Idleness inherits the classical stable Ornstein--Uhlenbeck tail; the queue has a sharp exponential tail with a polynomial prefactor whose power is half the stable index and a positive constant represented through regeneration.
Conditional on small idleness, idleness divided by the conditioning threshold and the queue converge to independent beta and size-biased limits.
Under large-queue conditioning, normalized idleness and queue excess converge to independent Mittag--Leffler and exponential limits.

\medskip
\noindent\textbf{Keywords:} reflected stochastic differential equation, stable L\'evy process, oblique reflection, exponential ergodicity, invariant distribution, join-the-shortest-queue

\smallskip
\noindent\textbf{MSC classification:} Primary: 60H10, 60J25; secondary: 60G52, 60K25
\end{abstract}

\section{Introduction}\label{sec:intro}

In large-scale parallel server queueing networks, 
join-the-shortest-queue (JSQ) routing sends each arriving customer to a server with the fewest customers.
Customers remain at their assigned server: a waiting customer is not transferred when another server becomes idle.
Idle servers and waiting customers can therefore coexist, although arrivals always use an idle server when one is available.
Understanding steady-state congestion thus requires a joint description of idleness and waiting.
For parallel exponential servers with Poisson arrivals, Eschenfeldt and Gamarnik \cite{EschenfeldtGamarnik2018} give such a description via a two-dimensional reflected diffusion in the Halfin--Whitt regime.

Heavy-tailed interarrival times introduce a different source of variability.
A long arrival gap allows many services to finish and can create substantial idleness.
When the interarrival-time tail is regularly varying with index $-\alpha$, where $1<\alpha<2$, stable limits for renewal processes motivate an $\alpha$-stable L\'evy process in place of Brownian noise (\cite{Whitt2002}, Section~6.3).
In the idleness equation, positive jumps represent increases in idle capacity caused by long arrival gaps.
We study a reflected stochastic differential equation (SDE) with this stable input and positive spare-capacity parameter $\beta$.

The routing rule determines how arrival variability affects the two state variables.
Let $I$ denote the scaled number of idle servers and $Q$ the scaled number of servers with at least one waiting customer.
When an idle server is available, arrivals enter service immediately rather than queue.
Arrival fluctuations therefore change $I$ without adding to $Q$, while service completions reduce existing queues.
At the scale motivating our model, service-completion fluctuations vanish, so $Q$ decays deterministically as long as $I>0$.
When no server is idle, further arrivals must wait.
In the limiting SDE, the nondecreasing regulator $U$ can increase only at $I=0$, offsetting any decrease that would make $I$ negative and adding the same amount to $Q$.
Its increments act in direction $(1,1)$, rather than the inward normal $(1,0)$ to the boundary $I=0$, making the reflection \textit{oblique}.
Arrival variability thus affects queue growth through the regulator, although the queue equation has no separate noise term.
We study how quickly this process approaches equilibrium and how arrival fluctuations and reflection shape its joint stationary law.

\subsection{Main results and approach}\label{sec:contributions}

We prove positive Harris recurrence and exponential convergence to equilibrium.
Our convergence estimate allows observables with polynomial growth in idleness and exponential growth in the queue: it holds in the weighted total variation norm for $(1+i)^p\e^{\theta q}$ for every $0<p<\alpha$ and $0<\theta<\theta_c$, where $\theta_c$ is the positive zero of the input's Laplace exponent after adding the spare-capacity drift.
These thresholds also define the exact stationary polynomial--exponential moment domain (\Cref{thm:exponential-ergodicity,thm:invariant-moments}).

We construct a Lyapunov function that combines polynomial growth in idleness with the oblique derivative condition.
Our fixed-time minorization uses boundary contact to transmit input randomness to the queue coordinate; two selected jump sizes then control the two terminal coordinates.
Together with attraction of the zero-input dynamics, these estimates verify the hypotheses of continuous-time Harris theory \cite{MeynTweedie1993Resolvents,DownMeynTweedie1995}.

We use stationary queue balance to identify the regulator occupation measure, whose normalization is the size-biased stationary queue law.
A separate idleness-weighting identity describes how observations change when weighted by idle capacity (\Cref{thm:idleness-weighting}).
The regulator occupation identity expresses the boundary term in the joint-transform equation through the queue marginal.
We use the idleness-weighting identity to obtain a queue density, then compare the stationary process with its reflection-free evolution and use the density of a stable stochastic convolution to prove joint absolute continuity (\Cref{prop:joint-stationary-density}).
Solving the joint-transform equation recovers the joint stationary law from the queue marginal (\Cref{thm:joint-transform-reconstruction}).

A complementary pathwise construction realizes a stationary state from a stable path on a fixed finite interval.
Multiplying the queue coordinate by an exponential factor and changing time removes its linear decay.
Successive integration and running-supremum operations then give a convergent recursion (\Cref{thm:stationary-representation}).
Using the same stable path for different values of spare capacity $\beta$ couples the stationary queue laws: the queue is stochastically nonincreasing in  $\beta$, and its mean is convex in $\beta$.
Thus equal increases in spare capacity yield nonincreasing reductions in the mean queue coordinate (\Cref{prop:queue-parameter-dependence}).

The tail analysis shows how input asymmetry appears in stationarity.
A convolution representation separates a stable Ornstein--Uhlenbeck (OU) term from an exponentially integrable correction.
Idleness therefore inherits the power-law tail of the stable OU law, with constant $c_{\mathrm A}/\alpha^2$ under our normalization \cite{Sato1999,HultLindskog2006}.
In contrast, the queue tail is asymptotic to $C_\pi x^{\alpha/2}\e^{-\theta_c x}$ (\Cref{thm:marginal-tails}).
The Laplace exponent of the input with spare-capacity drift determines the exponential rate, while the oblique derivative condition determines the polynomial power.
A regenerative decomposition represents the positive constant $C_\pi$.
Conditioning on a large queue yields independent Mittag--Leffler and exponential limits for rescaled idleness and queue excess, with typical idleness smaller than its conditional-mean scale (\Cref{cor:large-queue-idleness}).

Near the reflecting face, conditioning on $I\le\varepsilon$ yields a joint limit: $I/\varepsilon$ becomes asymptotically independent of a queue with the size-biased stationary law as $\varepsilon\downarrow0$ (\Cref{thm:boundary-asymptotics}).
The size-biased queue law also arises by normalizing the regulator occupation measure, linking boundary conditioning to stationary balance identities.

\subsection{Related literature}\label{sec:related-literature}

The Brownian JSQ diffusion is the closest benchmark.
Our nonnegative idleness coordinate is the negative of the centered busy-server coordinate in the Brownian formulation \cite{EschenfeldtGamarnik2018}.
Braverman \cite{Braverman2020} proves exponential ergodicity of this diffusion and tightness of the scaled stationary distributions of the queueing systems, using drift-based fluid limit Lyapunov functions within the generator expansion framework.
Banerjee and Mukherjee \cite{BanerjeeMukherjee2019} construct regenerative cycles and obtain Gaussian-type bounds for the idleness tail and exponential-type bounds for the queue tail.
Their subsequent work studies stationary moments and sensitivity to the spare-capacity parameter \cite{BanerjeeMukherjee2020}.
The regulator and mixed-moment relations in that work, together with exponential integrability of the queue, imply the Brownian counterparts of our regulator-occupation and idleness-weighting identities.
These identities describe structure shared by the Brownian and stable-driven equations.

Our stationary construction is related to the backward construction of Loynes \cite{Loynes1962} and to the correspondence between self-similarity and stationarity associated with Lamperti \cite{Lamperti1962}.
Transform methods and asymptotic analysis for L\'evy-driven queues are surveyed and developed by D{\k{e}}bicki and Mandjes \cite{DebickiMandjes2012,DebickiMandjes2015}.
Kella \cite{Kella1996} studies stability and nonproduct stationary distributions for fluid networks with L\'evy inputs.
Konstantopoulos et~al. \cite{KonstantopoulosLastLin2004} establish stability and stationary equations for multidimensional L\'evy networks with oblique reflection.

For unreflected L\'evy-driven OU processes, Wang \cite{Wang2012} gives conditions for exponential ergodicity.
Arapostathis et~al. \cite{ArapostathisPangSandric2019} identify polynomial and exponential convergence regimes for piecewise OU processes arising from many-server queues, including sharp polynomial rates in the absence of abandonment under their assumptions.
Related controlled models admit polynomial convergence bounds uniform over suitable classes of stationary Markov controls \cite{ArapostathisHmediPangSandric2019uniform}.
For a one-dimensional reflected OU process with spectrally positive L\'evy input, Xing et~al. \cite{XingZhangWang2009} obtain stationary-distribution and transform results; their existence theorem assumes a nonzero Brownian component.

For constrained Brownian diffusions, Atar et~al. \cite{AtarBudhirajaDupuis2001} give sufficient conditions for positive recurrence, whereas Kang and Ramanan \cite{KangRamanan2014} characterize stationary distributions through a submartingale formulation.
Quantitative convergence results are also available for reflected processes.
Sarantsev \cite{Sarantsev2016} obtains explicit exponential convergence rates for reflected jump-diffusions on the half-line under conditions including exponential moments of the jumps.
Huang and Pang \cite{HuangPang2025} establish exponential convergence in total variation and Wasserstein distance for a two-dimensional OU process with normal reflection and nondegenerate Brownian input.
Our process instead has pure stable noise acting only on idleness and a common oblique regulator.

More broadly, this work relates to the ergodicity of degenerate SDEs, where random forcing acts only in certain directions.
For uniformly elliptic diffusions, regularity and Harnack estimates provide routes to transition-density bounds used to verify irreducibility and minorization.
When the noise is degenerate, estimates based on uniform ellipticity no longer apply directly, and a central question is how the dynamics transmit randomness to the unforced coordinates.
In many unreflected models, coupling through the drift provides this transmission: perturbations of the forced variables influence the remaining coordinates as the system evolves.
Control-theoretic approaches to recurrence and invariant measures for degenerate diffusions are developed by Kliemann \cite{Kliemann1987} and Arnold and Kliemann \cite{ArnoldKliemann1987}.

For hypoelliptic particle models, Mattingly and Stuart \cite{MattinglyStuart2002} combine Lyapunov estimates with transition-density and accessibility arguments to prove geometric ergodicity.
Mattingly et~al. \cite{MattinglyStuartHigham2002} develop related arguments for SDEs and analyze when numerical approximations preserve ergodicity.
These approaches combine control of large excursions with local estimates ensuring that trajectories from different initial states can reach common regions of the state space.

The propagation of randomness is also central to stochastic fluid dynamics with degenerate forcing.
Weinan E and Jonathan C. Mattingly \cite{EMattingly2001} prove exponential ergodicity for finite-dimensional Galerkin approximations of the two-dimensional Navier--Stokes equations.
For the infinite-dimensional equations, Hairer and Mattingly \cite{HairerMattingly2006} establish unique ergodicity on the relevant invariant subspace using the asymptotic strong Feller property and a weak form of irreducibility.
For unreflected degenerate jump-driven SDEs, related regularity questions are studied under H\"ormander-type conditions by Zhang \cite{Zhang2014} and, using a parametrix approach, by Huang and Menozzi \cite{HuangMenozzi2016}.

In our model, reflection rather than the interior drift transmits randomness to the queue coordinate.
While the regulator is inactive, the queue decays deterministically; randomness reaches that coordinate through regulator increments at the reflecting face.
Our minorization construction makes this transmission quantitative: boundary contact followed by reflection-free evolution yields a common lower bound with a two-dimensional density.
The accompanying Lyapunov argument accommodates both the heavy-tailed jumps and the oblique derivative condition.

\subsection{Organization and notation}
\Cref{sec:model} specifies the SDE and states the main results.
\Cref{sec:lyapunov,sec:harris} develop the Lyapunov and minorization arguments, and \Cref{sec:stationary-law} characterizes the invariant law.
\Cref{sec:tails} establishes the tail and boundary asymptotics, and \Cref{sec:discussion} concludes.
\Cref{app:fixed-time-minorization} gives the pathwise comparison and minorization proofs, and \Cref{app:queue-tail-estimates} gives the supporting queue-tail estimates.
\Cref{app:pathwise,app:stable-input,app:stationary-localization,app:further-invariant-law} contain the well-posedness and stable-input details, localization arguments, and further properties of the invariant law.

We write $E=\R_+^2$, where $\R_+=[0,\infty)$, set $x^+=\max\{x,0\}$, and use the Euclidean norm unless the $\ell^1$ norm $\|(i,q)\|_1=|i|+|q|$ is specified.
We write $A\Subset B$ when $\overline A$ is compact and contained in $B$.
For $S\subseteq\R^d$, the space $\Dspace([0,T],S)$ consists of c\`adl\`ag paths with values in $S$, and $\|x\|_T=\sup_{0\le s\le T}\|x(s)\|$ denotes the uniform norm.
We use $\Longrightarrow$ for convergence in distribution.
By $C^2(E)$ we mean restrictions to $E$ of $C^2$ functions on a neighborhood of $E$; derivatives at the boundary are their one-sided limits from the interior.
Constants may depend on the fixed model parameters; additional dependencies are indicated when relevant.
For a measure $\mu$ and a $\mu$-integrable function $f$, write $\mu(f):=\int f\,\dd\mu$.
For positive functions, $f\sim g$ means that $f/g\to1$ in the stated limit, whereas $f\asymp g$ means that their ratio is bounded above and below by positive constants; subscripts on $\asymp$ indicate which parameters the comparison constants may depend on.

\section{The reflected SDE and main results}\label{sec:model}

\subsection{Dynamics and pathwise properties}

Fix $\alpha\in(1,2)$, $\beta>0$, and $c_{\mathrm A}>0$.
Let $\Stable$ be a spectrally positive strictly $\alpha$-stable L\'evy process started at zero, with L\'evy measure
\begin{equation*}
\nu_{\mathrm A}(\dd z)
=c_{\mathrm A}z^{-1-\alpha}\1_{(0,\infty)}(z)\dd z.
\end{equation*}
The coefficient $c_{\mathrm A}$ sets the stable input scale.
Since $\alpha>1$, the input has a finite first absolute moment, and strict stability implies $\E[\Stable(t)]=0$.
For an initial state $Y_0=(I_0,Q_0)\in E$, independent of future increments of $\Stable$, consider the reflected SDE for $Y=(I,Q)$:
\begin{equation}\label{eq:reflected-sde}
\begin{split}
I_t&=I_0+\Stable(t)+\int_0^t(\beta-I_s-Q_s)\dd s+U_t,\\
Q_t&=Q_0-\int_0^tQ_s\dd s+U_t.
\end{split}
\end{equation}
The regulator $U$ is nondecreasing, $U_0=0$, and
\begin{equation}\label{eq:regulator-condition}
I_t\ge0,
\qquad
\int_0^T\1_{\{I_t>0\}}\dd U_t=0,
\qquad T<\infty.
\end{equation}
It can increase only on the face $I=0$ and contributes equally to both coordinates, giving the reflection direction $(1,1)$.
The inward normal is $(1,0)$; the queue displacement makes the reflection oblique.

Two identities describe how reflection affects the state.
Variation of constants in the queue equation yields
\begin{equation}\label{eq:queue-variation-of-constants}
Q_t=\e^{-t}Q_0+\int_{(0,t]}\e^{-(t-s)}\dd U_s\ge0.
\end{equation}
Thus $Q$ decays exponentially while $U$ is constant, and regulator increments contribute through the exponential kernel.
Reflection acts only at $I=0$; the formula ensures nonnegativity of $Q$ without a second regulator at $Q=0$.
The common regulator cancels from the coordinate difference, giving
\begin{equation}\label{eq:regulator-cancellation}
I_t-Q_t
=I_0-Q_0+\Stable(t)+\int_0^t(\beta-I_s)\dd s.
\end{equation}
This gives the stationary mean-idleness balance and is used in the stationary construction and tail analysis.

Let $y=(i,q)\in E$ be a deterministic state and use the superscript $y$ to indicate a solution started from $y$.
\begin{proposition}[Well-posedness and Markov properties]\label[proposition]{prop:well-posedness}
For every initial state $y\in E$, \cref{eq:reflected-sde,eq:regulator-condition} have a pathwise unique, nonexplosive strong solution $Y^y=(I^y,Q^y)$.
The process is c\`adl\`ag, Feller--Dynkin, Borel right, and strong Markov.
Both $U$ and $Q$ are continuous.
\end{proposition}

For two solutions driven by the same input, the initial $\ell^1$ separation bounds both their later distance and the accumulated absolute difference in their idleness.
\begin{proposition}[Nonexpansiveness in the initial state]\label[proposition]{prop:initial-nonexpansiveness}
For any common c\`adl\`ag input starting at zero and having no negative jumps, the corresponding solutions satisfy
\begin{equation}\label{eq:initial-nonexpansive}
\|Y_t^y-Y_t^{\widetilde y}\|_1
+\int_0^t|I_s^y-I_s^{\widetilde y}|\dd s
\le\|y-\widetilde y\|_1,
\qquad t\ge0,\quad y,\widetilde y\in E.
\end{equation}
\end{proposition}
\Cref{app:pathwise} proves \Cref{prop:well-posedness} and gives the finite-horizon estimate for changes in the input path.
\Cref{app:initial-nonexpansiveness} proves \Cref{prop:initial-nonexpansiveness}.

\subsection{Invariant distribution and exponential convergence}

Mean reversion and reflection bring the process back to compact sets, which establishes positive Harris recurrence.
\begin{proposition}[Positive Harris recurrence]\label[proposition]{prop:positive-harris-recurrence}
The process $Y$ is positive Harris recurrent and has a unique invariant probability measure $\pi$.
\end{proposition}

Use $\Pp_\pi,\E_\pi$ for probabilities and expectations when $Y_0$ has law $\pi$, independently of future input increments.
In stationary probabilities and expectations without time subscripts, $(I,Q)$ denotes a random vector with law $\pi$; let $\pi_Q$ be the law of $Q$ and set $m_Q:=\E_\pi Q$.
The stationary balances in \Cref{prop:regulator-occupation} give $\E_\pi I=\beta$, $0<m_Q<\infty$, and $\E_\pi U_t=tm_Q$.
Thus the spare-capacity parameter fixes mean idleness, and the stationary mean of $Q$ equals the mean regulator growth rate.

To quantify convergence to $\pi$, we use polynomial control of idleness and exponential control of the queue.
These growth rates reflect the input's power-law positive jumps and finite exponential moments in the negative direction.
For $p,\theta\ge0$, define the weight function $W_{p,\theta}:E\to[1,\infty)$ by
\[
W_{p,\theta}(y):=(1+i)^p\e^{\theta q},\qquad y=(i,q)\in E.
\]
Since $\alpha>1$, the estimate gives exponential convergence of mean idleness, the mean queue coordinate, and the mixed moment $\E_y[I_tQ_t]$ to their stationary values.

Let $P_t(y,\cdot)$ be the transition law, let $\Pp_y,\E_y$ denote probability and expectation from $y$, and write $P_tf(y):=\E_y[f(Y_t)]$ for measurable $f$ when this expectation is well defined.
For a signed measure $\mu$ with $\int_E W_{p,\theta}(y)\,|\mu|(\dd y)<\infty$, where $|\mu|$ is its total variation measure, define the weighted total variation norm
\[
\|\mu\|_{p,\theta}:=\sup_{|h|\le W_{p,\theta}}|\mu(h)|,
\]
where the supremum is over measurable functions.
Since $W_{p,\theta}\ge1$, this norm controls both probabilities and expectations of functions $h$ satisfying $|h|\le W_{p,\theta}$.

The input's Laplace exponent determines the range of exponential weights:
\begin{equation}\label{eq:stable-laplace-exponent}
\log\E\e^{-\theta\Stable(t)}
=t\kappa_\alpha\theta^\alpha,
\qquad t,\theta\ge0,
\end{equation}
where $\kappa_\alpha:=c_{\mathrm A}\Gamma(2-\alpha)/[\alpha(\alpha-1)]>0$.
Including the spare-capacity drift gives
\[
\log\E\e^{-\theta(\Stable(t)+\beta t)}
=t\Phi(\theta),
\quad \text{where} \quad 
\Phi(\theta):=\kappa_\alpha\theta^\alpha-\beta\theta.
\]
This exponent $\Phi(\theta)$ is negative precisely for $0<\theta<\theta_c$, where
\[
\theta_c:=\left(\frac{\beta}{\kappa_\alpha}\right)^{1/(\alpha-1)}
\]
is its unique positive root.
The next result gives exponential convergence for queue weights $\e^{\theta q}$ in this range.

\begin{theorem}[Exponential ergodicity in weighted total variation]\label[theorem]{thm:exponential-ergodicity}
For every $p\in(0,\alpha)$ and $\theta\in(0,\theta_c)$, there are constants $M_{p,\theta}<\infty$ and $\rho_{p,\theta}>0$ such that
\[
\|P_t(y,\cdot)-\pi\|_{p,\theta}
\le M_{p,\theta}W_{p,\theta}(y)\e^{-\rho_{p,\theta}t},
\qquad y\in E,\quad t\ge0.
\]
\end{theorem}

\subsection{Construction of the invariant law}

The invariant law can be constructed from a specified stable input by integration and running-supremum operations.
Let $\widehat\Stable$ have the same law as $\Stable$ and set $A(x):=\beta x+\widehat\Stable(x^\alpha/\alpha)$ for all $x\ge0$.
Starting from $R_0\equiv0$, define, for $n\ge0$,
\begin{equation}\label{eq:stationary-recursion}
J_n(x):=\int_0^x\frac{R_n(v)}v\dd v,
\qquad
R_{n+1}(x):=\sup_{0\le u\le x}[J_n(u)-A(u)]^+.
\end{equation}
At each step, $J_n$ supplies the accumulated queue contribution, and $R_{n+1}$ is the least nondecreasing correction that keeps $A-J_n+R_{n+1}$ nonnegative.
The integrals at zero are improper; their finiteness is part of the result.

\begin{theorem}[Invariant law from stable input]\label[theorem]{thm:stationary-representation}
Almost surely, $R_n,J_n$ are finite and continuous, are nondecreasing in $n$, and converge uniformly on compact intervals to finite limits $R_\infty,J_\infty$.
Moreover,
\begin{equation}\label{eq:stationary-law-representation}
\pi=\law\bigl(A(1)+R_\infty(1)-J_\infty(1),\ R_\infty(1)\bigr).
\end{equation}
Only the stable path on $[0,1/\alpha]$ is needed to construct this random vector.
\end{theorem}

In the stationary construction, the time change $x=\e^t$ identifies $R_\infty(x)$ with $xQ_{\log x}$.
This change of time encodes the stationary past in the finite input interval; \Cref{sec:stationary-representation} develops the reconstruction.
For fixed input parameters, coupling the recursion across values of $\beta$ with one stable path shows that the stationary queue is stochastically nonincreasing in $\beta$ and has a convex mean in $\beta$ (\Cref{prop:queue-parameter-dependence}).

\subsection{Stationary tails, moments, and boundary limits}

The following tail equivalents describe the contrasting stationary extremes of idleness and the queue.
\begin{theorem}[Sharp stationary tail asymptotics]\label[theorem]{thm:marginal-tails}
For some $C_\pi\in(0,\infty)$, under $\pi$, as $x\to\infty$,
\begin{subequations}
\begin{align}
\Pp_\pi(I>x)&\sim\Pp_\pi(Q-I<-x)
\sim\frac{c_{\mathrm A}}{\alpha^2}x^{-\alpha},
\label{eq:idleness-tail}\\
\Pp_\pi(Q>x)&\sim\Pp_\pi(Q-I>x)
\sim C_\pi x^{\alpha/2}\e^{-\theta_cx}.
\label{eq:queue-sharp-tail}
\end{align}
\end{subequations}
\end{theorem}

A positive input jump increases idleness immediately, so a single large jump can produce extreme idleness.
The queue grows only through reflection at $I=0$, so its large values arise from sustained negative input fluctuations that repeatedly activate the regulator.
For the queue tail, \Cref{sec:queue-tail-analysis} uses upcrossings at $I=0$ to construct regenerative cycles.
The stationary tail factors into a hitting probability and an expected occupation time, divided by the mean cycle length.
The occupation factor converges to a positive constant, so the main task is to estimate the hitting probability sharply.
We obtain upper and lower bounds by correcting $\e^{\theta_c(q-i)}$ to satisfy the required generator inequalities and the oblique derivative condition.
Compatibility of the generator cancellations with this boundary condition determines the polynomial factor $x^{\alpha/2}$, while regeneration gives $C_\pi$ in \cref{eq:queue-tail-constant}.

Setting $\theta=0$ in $\E_\pi W_{p,\theta}(I,Q)$ gives polynomial idleness moments; setting $p=0$ gives queue exponential moments.
Together with the subcritical joint moment bounds, these tail equivalents identify the exact polynomial--exponential moment domain.
\begin{theorem}[Exact polynomial--exponential moment domain]\label[theorem]{thm:invariant-moments}
For all $p,\theta\ge0$,
\begin{equation}\label{eq:exact-invariant-moment-domain}
\E_\pi[(1+I)^p\e^{\theta Q}]<\infty
\quad\Longleftrightarrow\quad
p<\alpha\ \text{ and }\ \theta<\theta_c.
\end{equation}
\end{theorem}

Near the reflecting face, the boundary asymptotic describes small idleness and a size-biased limit for the conditional queue distribution.
Define the size-biased queue variable $Q^{\mathrm{sb}}$ by
\[
\Pp(Q^{\mathrm{sb}}\in\dd q)=\frac{q\,\pi_Q(\dd q)}{m_Q}.
\]
This is a probability law because $m_Q=\E_\pi Q\in(0,\infty)$.
Let $B\sim\operatorname{Beta}(\alpha-1,1)$ be independent of $Q^{\mathrm{sb}}$.

\begin{theorem}[Small-idleness asymptotics and conditional limit]\label[theorem]{thm:boundary-asymptotics}
The small-idleness asymptotic is
\begin{equation}\label{eq:small-idleness-asymptotic}
\lim_{\varepsilon\downarrow0}\varepsilon^{1-\alpha}\Pp_\pi(I\le\varepsilon)
=\frac{m_Q}{\kappa_\alpha\Gamma(\alpha)}.
\end{equation}
Moreover,
\begin{equation}\label{eq:boundary-product-limit}
\law_\pi\left(\frac I\varepsilon,Q\ \middle|\ I\le\varepsilon\right)
\Longrightarrow\law(B,Q^{\mathrm{sb}}),\qquad \varepsilon\downarrow0.
\end{equation}
\end{theorem}

The two idleness asymptotics determine the range of finite moments and their divergence rates at its endpoints.
\begin{corollary}[Positive and negative moments of idleness]\label[corollary]{cor:real-idleness-moments}
For every $r\in\R$,
\begin{equation}\label{eq:real-idleness-moments}
\E_\pi I^r<\infty
\quad\Longleftrightarrow\quad
1-\alpha<r<\alpha.
\end{equation}
At the endpoints of this interval, the moments diverge according to
\begin{subequations}
\begin{align}
\lim_{a\uparrow\alpha-1}(\alpha-1-a)\E_\pi I^{-a}
&=\frac{m_Q}{\kappa_\alpha\Gamma(\alpha-1)},\label{eq:inverse-idleness-pole}\\
\lim_{r\uparrow\alpha}(\alpha-r)\E_\pi I^r
&=\frac{\cA}{\alpha}.\label{eq:positive-idleness-pole}
\end{align}
\end{subequations}
\end{corollary}
In particular, $\E_\pi I=\beta<\infty$, whereas $\E_\pi I^2=\E_\pi I^{-1}=\infty$.
The power-law upper tail makes the second moment diverge; the inverse moment diverges because $\Pp_\pi(I\le\varepsilon)$ is of order $\varepsilon^{\alpha-1}$, even though $\Pp_\pi(I=0)=0$.
\Cref{sec:tails} gives the proofs and consequences, including idleness conditional on a large queue.
\Cref{app:moment-refinements} gives the joint moment domain for weights singular at the reflecting face.

\section{Lyapunov drift estimates}\label{sec:lyapunov}

Lyapunov methods for reflected processes give tools for recurrence and stability \cite{DupuisWilliams1994,AtarBudhirajaDupuis2001}.
We prove a drift inequality for functions comparable to $W_{p,\theta}$.
We use this inequality together with the minorization argument in \Cref{sec:harris} to prove recurrence, stationary moment bounds, and exponential convergence.

\subsection{The generator and the oblique derivative condition}\label{sec:generator}

A positive jump of size $z$ sends $(i,q)$ to $(i+z,q)$.
Let $\mathcal C_J^2(E)$ be the class of functions $f\in C^2(E)$ such that
\[
 \int_0^\infty\sup_{(i,q)\in K}
|f(i+z,q)-f(i,q)-zf_i(i,q)|z^{-1-\alpha}\dd z<\infty
\]
for every compact $K\subset E$.
For $f\in\mathcal C_J^2(E)$, define the integro-differential operator
\begin{equation} \label{eq:gen-jumpgen}
\begin{split}
\gen f(i,q)&:=(\beta-i-q)f_i(i,q)-qf_q(i,q)+\jumpgen f(i,q),\\
\jumpgen f(i,q)&:=c_{\mathrm A}\int_0^\infty
[f(i+z,q)-f(i,q)-zf_i(i,q)]z^{-1-\alpha}\dd z.
\end{split}
\end{equation}
Subtracting $zf_i(i,q)$ gives first-order compensation near zero and centers the large jumps.
\Cref{app:stable-normalization} gives the equivalent representation with compensation restricted to small jumps.

A regulator increment acts along $(1,1)$, so its contribution to It\^o's formula is the oblique derivative $f_i(0,q)+f_q(0,q)$.
Define the subclass satisfying the oblique derivative condition
\[
\mathcal C_{\mathrm{ref}}^2(E):=
\left\{f\in\mathcal C_J^2(E):
 f_i(0,q)+f_q(0,q)=0,\ q\ge0\right\}.
\]
For functions in $\mathcal C_{\mathrm{ref}}^2(E)$, the regulator term vanishes.

Let $\tau_m=\inf\{t\ge0:I_t+Q_t\ge m\}$.
For $f\in\mathcal C_{\mathrm{ref}}^2(E)$, It\^o's formula gives
\begin{equation}\label{eq:stopped-dynkin-formula}
f(Y_{t\wedge\tau_m})-f(Y_0)
=\int_0^{t\wedge\tau_m}\gen f(Y_s)\dd s+M_t^{f,m},
\end{equation}
where $M^{f,m}$ is a martingale for each deterministic initial state.
Thus $f$ is in the extended-generator domain in the local-martingale sense, with extended generator $\gen f$.
\Cref{prop:generator-criterion} justifies the stopped formula, including exits caused by jumps; \Cref{prop:stationary-ito} treats stationary test functions with a nonzero oblique derivative.

We write
\[
\jumpgenI g(i):=c_{\mathrm A}\int_0^\infty
[g(i+z)-g(i)-zg'(i)]z^{-1-\alpha}\dd z
\]
for the jump operator on functions of idleness alone.

\subsection{Construction of the Lyapunov function}

For a product-form function $V(i,q)=\e^{\theta q}g(i)$, the oblique derivative condition requires $g'(0)=-\theta g(0)$.
The factor $\e^{-\theta i}$ supplies this boundary derivative.
A correction whose value and derivative vanish at zero can add polynomial growth in idleness without changing the boundary condition.
For $0\le p<\alpha$, define
\begin{equation}\label{eq:polynomial-idleness-correction}
\psi_p(i):=i^2(1+i)^{p-2},\qquad i\ge0.
\end{equation}
Then $\psi_p(0)=\psi_p'(0)=0$ and $\psi_p'(i)=i(1+i)^{p-3}(2+pi)>0$ for $i>0$.
For $p>0$, $\psi_p(i)\sim i^p$; the bounded function $\psi_0$ will be used in the queue-tail estimates at $\theta=\theta_c$.

\begin{lemma}[Jump-operator estimates]\label[lemma]{lem:jump-operator-estimates}
For $0\le p<\alpha$, the function $\jumpgenI\psi_p$ is continuous, the map $(i,q)\mapsto\psi_p(i)$ belongs to $\mathcal C_J^2(E)$, and
\begin{equation}\label{eq:polynomial-jump-estimates}
0<\jumpgenI\psi_p(0)<\infty,
\qquad
\jumpgenI\psi_p(i)=O(i^{p-\alpha})\quad(i\to\infty).
\end{equation}
\end{lemma}
\begin{proof}
Fix $0\le p<\alpha$.
At zero, the beta integral gives
\[
\jumpgenI\psi_p(0)
=\cA\int_0^\infty z^{1-\alpha}(1+z)^{p-2}\dd z
=\cA\frac{\Gamma(2-\alpha)\Gamma(\alpha-p)}{\Gamma(2-p)},
\]
which is finite and strictly positive.

From the definition of $\psi_p$ and direct differentiation, for $i\ge0$ we have
\[
0\le\psi_p(i)\le(1+i)^p,\qquad
|\psi_p'(i)|\le C(1+i)^{p-1},\qquad
|\psi_p''(i)|\le C(1+i)^{p-2}.
\]
Uniformly for $i$ in a compact interval, Taylor's theorem therefore bounds the absolute jump remainder by $Cz^2$ for $0<z\le1$, while the bounds on $\psi_p$ and $\psi_p'$ give $C(1+z+z^p)$ for $z>1$.
These bounds are integrable against $z^{-1-\alpha}\dd z$ because $p<\alpha$ and $1<\alpha<2$.
Thus $(i,q)\mapsto\psi_p(i)$ belongs to $\mathcal C_J^2(E)$, and dominated convergence gives continuity of $\jumpgenI\psi_p$.

To estimate $\jumpgenI\psi_p(i)$ as $i\to\infty$, fix $i\ge2$ and split the integral at $z=i/2$.
For $0<z\le i/2$, the bound $|\psi_p''(i+u)|\le Ci^{p-2}$ for $0\le u\le z$ gives a contribution of at most $Ci^{p-2}\int_0^{i/2}z^{1-\alpha}\dd z=O(i^{p-\alpha})$.
For $z>i/2$, the absolute jump remainder is at most $C[(i+z)^p+i^p+zi^{p-1}]$.
The substitution $z=iu$, together with $p<\alpha$ and $\alpha>1$, bounds the integral over this range by $O(i^{p-\alpha})$ as well.
Combining the two contributions proves the asserted bound. \hfill
\end{proof}

Fix $0<p<\alpha$ and $0<\theta<\theta_c$.
Since $\Phi(\theta)<0$, set
\begin{equation*}
\varepsilon_{p,\theta}:=-\frac{\Phi(\theta)}{2\jumpgenI\psi_p(0)}>0,
\qquad g_{p,\theta}(i):=\e^{-\theta i}+\varepsilon_{p,\theta}\psi_p(i),
\qquad V_{p,\theta}(i,q):=\e^{\theta q}g_{p,\theta}(i).
\end{equation*}
This choice of $\varepsilon_{p,\theta}$ preserves negative generator drift at the reflecting face.
The resulting function satisfies
\[
g_{p,\theta}(0)=1,\qquad g_{p,\theta}'(0)=-\theta,\qquad
(V_{p,\theta})_i(0,q)+(V_{p,\theta})_q(0,q)=0.
\]
Moreover,
\begin{equation}\label{eq:lyapunov-weight-comparison}
g_{p,\theta}(i)\sim\varepsilon_{p,\theta}i^p\quad(i\to\infty),\qquad
V_{p,\theta}(i,q)\asymp_{p,\theta}(1+i)^p\e^{\theta q}\quad(i,q)\in E.
\end{equation}
Thus $V_{p,\theta}$ has compact sublevel sets and a strictly positive minimum. 

\begin{proposition}[Foster--Lyapunov inequality]\label[proposition]{prop:foster-lyapunov}
For every $0<p<\alpha$ and $0<\theta<\theta_c$, there are a compact rectangle $K_{p,\theta}$ and constants $\delta_{p,\theta},b_{p,\theta}>0$ such that
\[
\gen V_{p,\theta}\le-\delta_{p,\theta}V_{p,\theta}
+b_{p,\theta}\1_{K_{p,\theta}}.
\]
\end{proposition}
\begin{proof}
For the exponential factor, the compensated jump remainder satisfies
\[
0\le \e^{-\theta(i+z)}-\e^{-\theta i}+\theta z\e^{-\theta i}
\le
\begin{cases}
\theta^2z^2/2,&0<z\le1,\\
\theta z,&z>1,
\end{cases}
\qquad i\ge0.
\]
The first bound follows from Taylor's theorem, and the second from $\e^{-\theta z}\le1$.
Both bounds are integrable against $z^{-1-\alpha}\dd z$ because $1<\alpha<2$.
Together with \Cref{lem:jump-operator-estimates} and boundedness of $\e^{\theta q}$ on compact $q$-intervals, they show that $V_{p,\theta}\in\mathcal C_J^2(E)$.
Since $(V_{p,\theta})_i(0,q)+(V_{p,\theta})_q(0,q)=0$, we have $V_{p,\theta}\in\mathcal C_{\mathrm{ref}}^2(E)$, so \Cref{prop:generator-criterion} applies.

Writing $g=g_{p,\theta}$ and using $\jumpgenI\e^{-\theta i}=\kappa_\alpha\theta^\alpha\e^{-\theta i}$, differentiation gives
\begin{equation*}
\frac{\gen V_{p,\theta}}{V_{p,\theta}}(i,q)
=\frac{(\beta-i)g'(i)+\jumpgenI g(i)}{g(i)}
-q\frac{\varepsilon_{p,\theta}[\psi_p'(i)+\theta\psi_p(i)]}{g(i)}.
\end{equation*}
The first term is continuous in $i$, equals $\Phi(\theta)/2<0$ at zero, and tends to $-p$ as $i\to\infty$, by \cref{eq:polynomial-jump-estimates}.
This term controls the Lyapunov drift near $i=0$ and for large $i$.
Choose $\delta>0$ and $0<i_0<i_1$ so that it is at most $-\delta$ on $[0,i_0]\cup[i_1,\infty)$.
Since the coefficient of $-q$ is nonnegative, the full ratio is also at most $-\delta$ in these two regions, uniformly over $q\ge0$.

It remains to control $i\in[i_0,i_1]$ when $q$ is large.
On this interval, the first term is bounded above, whereas the coefficient of $-q$ has a strictly positive minimum.
Thus there is $q_1>0$ such that the ratio is at most $-\delta$ for every $i\in[i_0,i_1]$ and $q\ge q_1$.
Together, these bounds give $\gen V_{p,\theta}\le-\delta V_{p,\theta}$ outside $K_{p,\theta}:=[0,i_1]\times[0,q_1]$.
Finally, continuity bounds $(\gen+\delta)V_{p,\theta}$ on this rectangle.
Taking $b_{p,\theta}>0$ larger than this bound and setting $\delta_{p,\theta}=\delta$ proves the inequality. \hfill
\end{proof}

\section{Harris recurrence and exponential ergodicity}\label{sec:harris}

The Lyapunov inequality supplies the drift condition for Harris theory.
We now establish a fixed-time minorization and use attraction of the zero-input dynamics to verify smallness, irreducibility, and aperiodicity.

\subsection{Deterministic attraction and fixed-time minorization}\label{sec:deterministic-minorization}

Let $\varphi_t(y)$ be the reflected solution with zero input.
Its equilibrium is $(\beta,0)$.
\begin{lemma}[Attraction of the zero-input dynamics]\label[lemma]{lem:deterministic-attraction}
For every compact $K\subset E$,
\[
\lim_{t\to\infty}\sup_{y\in K}\norm{\varphi_t(y)-(\beta,0)}=0.
\]
\end{lemma}

\begin{proof}
Fix $y=(i,q)\in E$ and write $(I_t,Q_t)=\varphi_t(y)$ and $\tau:=\inf\{t\ge0:I_t=0\}$.
For $i>0$, the regulator is constant before $\tau$, and
\[
Q_t=q\e^{-t},
\qquad
I_t=\beta+(i-\beta)\e^{-t}-qt\e^{-t}
\]
holds up to the first boundary hit, or for all $t\ge0$ if $\tau=\infty$.
A finite boundary hit from the interior has $Q_\tau\ge\beta$.
If $i=0$ and $q\le\beta$, the regulator is inactive and the same formula applies for all $t\ge0$.

In the remaining cases with $\tau<\infty$, including $\tau=0$ when $i=0$ and $q>\beta$, the state stays on $I=0$ with $\dot Q=-\beta$ until time $\sigma:=\tau+(Q_\tau-\beta)/\beta$.
At that time $(I_\sigma,Q_\sigma)=(0,\beta)$, and $Q_{\sigma+s}=\beta\e^{-s}$ and $I_{\sigma+s}=\beta[1-(1+s)\e^{-s}]$ for $s\ge0$.
These formulas give pointwise convergence to $(\beta,0)$.
By \Cref{prop:initial-nonexpansiveness}, the maps $\varphi_t$ are one-Lipschitz in the $\ell^1$ norm, uniformly in $t$.
A finite-net argument therefore gives convergence uniformly over compact initial sets. \hfill
\end{proof}

\begin{proposition}[Fixed-time minorization]\label[proposition]{prop:fixed-time-minorization}
Fix $K_0=[0,\bar i_0]\times[0,\bar q_0]$ with $\bar i_0>\beta$ and $\bar q_0>0$.
There exist $\Delta>0$, a nonempty open rectangle $\mathcal G\Subset(0,\infty)^2$, and $\varepsilon_0>0$ such that
\begin{equation}\label{eq:fixed-time-minorization-main}
P_\Delta(y,H)\ge\varepsilon_0\nu_{\mathcal G}(H),
\qquad y\in K_0,\quad H\in\mathcal B(E),
\qquad
\nu_{\mathcal G}(H):=\frac{\operatorname{Leb}_2(H\cap \mathcal G)}{\operatorname{Leb}_2(\mathcal G)}.
\end{equation}
\end{proposition}

\Cref{app:fixed-time-minorization} constructs a local minorization near $(\beta,0)$ using the L\'evy--It\^o decomposition.
With the small-jump martingale uniformly small and no large jump initially, the centering drift brings the reference trajectory to $I=0$.
The first selected jump size determines the terminal queue through a strictly decreasing bi-Lipschitz map.
The second jump starts a reflection-free interval and changes terminal idleness affinely without changing the queue.
Deterministic attraction, the small-ball bound \cref{eq:stable-small-ball}, and the pathwise estimate \cref{eq:fixed-time-solution-lipschitz} extend this local lower bound to $K_0$.

\subsection{Small sets, irreducibility, and aperiodicity}

Fix a choice of the minorization in \Cref{prop:fixed-time-minorization}.
A set $K$ is small if there are $t>0$ and a nonzero finite measure $\nu$ such that $P_t(y,\cdot)\ge\nu(\cdot)$ for every $y\in K$.
Such a set is also petite, since deterministic sampling is allowed in the definition of petiteness.

\begin{proposition}[Small sets, irreducibility, and aperiodicity]\label[proposition]{prop:smallness-irreducibility}
The process $Y$ is $\nu_{\mathcal G}$-irreducible and aperiodic, and every compact subset of $E$ is small.
\end{proposition}
\begin{proof}
Fix a compact $K\subset E$.
\Cref{lem:deterministic-attraction} gives $t_K<\infty$ and $r_K>0$ such that the zero-input solutions from $K$ stay at distance at least $r_K$ from $E\setminus K_0$ for $t\ge t_K$.
Let $C_t$ be the solution-map Lipschitz constant in \cref{eq:fixed-time-solution-lipschitz}.
That estimate and the small-ball bound \cref{eq:stable-small-ball} give
\[
a_K(t):=\inf_{y\in K}P_t(y,K_0)
\ge\Pp(\|\Stable\|_t<r_K/(2C_t))>0,
\qquad t\ge t_K.
\]
The semigroup property and \Cref{prop:fixed-time-minorization} therefore imply
\begin{equation}\label{eq:compact-time-minorization}
P_t(y,H)\ge\varepsilon_0a_K(t-\Delta)\nu_{\mathcal G}(H),
\qquad y\in K,\quad H\in\mathcal B(E),\quad t\ge t_K+\Delta.
\end{equation}
Choosing any such time proves smallness.
Taking $K=\{y\}$ gives positive transition probabilities at all sufficiently large times, and hence positive expected occupation, for every set of positive $\nu_{\mathcal G}$-measure.
Thus the process is $\nu_{\mathcal G}$-irreducible.
Attraction of the zero-input dynamics and the small-ball event also give positive probability of entering every rectangle $[0,M_I]\times[0,M_Q]$ with $M_I>\beta$ and $M_Q>0$ from every initial state.
These rectangles are compact, hence closed and petite.

For aperiodicity, choose a compact rectangle $K\supset\overline{\mathcal G}$.
It is small and satisfies $\nu_{\mathcal G}(K)=1$.
By \cref{eq:compact-time-minorization},
\[
P_t(y,K)\ge\varepsilon_0a_K(t-\Delta)>0,
\qquad y\in K,\quad t\ge t_K+\Delta.
\]
This verifies the continuous-time aperiodicity condition of Down et~al. \cite{DownMeynTweedie1995}. \hfill
\end{proof}

\subsection{Recurrence, moment bounds, and exponential convergence}

We now apply the continuous-time Harris theorem of Down et~al. (\cite{DownMeynTweedie1995}, Theorem~5.2\textup{(c)}) to prove positive Harris recurrence, stationary moment bounds, and exponential convergence.
Throughout this subsection, fix an arbitrary $p\in(0,\alpha)$ and $\theta\in(0,\theta_c)$, and let $V$ be a positive scalar multiple of $V_{p,\theta}$ chosen so that $V\ge1$.

\begin{proof}[Proof of \Cref{prop:positive-harris-recurrence}]
The process is nonexplosive and Borel right by \Cref{prop:well-posedness}, and irreducible and aperiodic by \Cref{prop:smallness-irreducibility}.
The drift inequality in \Cref{prop:foster-lyapunov} holds with a compact petite set.
\Cref{lem:lyapunov-dynkin} proves the unstopped Dynkin identity for $V$, so the inequality holds in the extended-generator sense required by Down et~al. (\cite{DownMeynTweedie1995}, Theorem~5.2\textup{(c)}).
That theorem gives positive Harris recurrence and a unique invariant probability measure $\pi$. \hfill
\end{proof}

We next record the required moment bound.
\begin{lemma}[Stationary polynomial--exponential moment bounds]\label[lemma]{lem:stationary-moment-bounds}
For $p\in(0,\alpha)$ and $\theta\in(0,\theta_c)$,
\begin{equation}\label{eq:stationary-moment-bound}
\E_\pi[(1+I)^p\e^{\theta Q}]<\infty.
\end{equation}
\end{lemma}

\begin{proof}
The hypotheses verified in the proof of \Cref{prop:positive-harris-recurrence} hold for every such pair.
The integrability conclusion of Down et~al. (\cite{DownMeynTweedie1995}, Theorem~5.2\textup{(c)}) therefore gives $\pi(V)<\infty$, where uniqueness identifies the same invariant law $\pi$ for every pair.
The comparison \cref{eq:lyapunov-weight-comparison} yields \cref{eq:stationary-moment-bound}.\hfill
\end{proof}

\begin{proof}[Proof of \Cref{thm:exponential-ergodicity}]
With the hypotheses verified in the proof of \Cref{prop:positive-harris-recurrence}, Down et~al. (\cite{DownMeynTweedie1995}, Theorem~5.2\textup{(c)}) give constants $M,\rho>0$ such that
\[
\|P_t(y,\cdot)-\pi\|_V\le MV(y)\e^{-\rho t},
\qquad y\in E,\quad t\ge0,
\]
where $\|\mu\|_V:=\sup_{|f|\le V}|\mu(f)|$.
By \cref{eq:lyapunov-weight-comparison}, there are constants $c,C>0$ such that $cW_{p,\theta}\le V\le CW_{p,\theta}$.
Thus $\|\mu\|_{p,\theta}\le c^{-1}\|\mu\|_V$, and the asserted bound follows with $M_{p,\theta}=MC/c$ and $\rho_{p,\theta}=\rho$.\hfill
\end{proof}

\section{Structure of the invariant law}\label{sec:stationary-law}

In this section, $Y_t=(I_t,Q_t)$ starts from $Y_0\sim\pi$, independently of future stable increments, and $U$ is its regulator.
We use stationary balance identities to study dependence between $I$ and $Q$, prove absolute continuity, and reconstruct the joint Laplace transform from the queue marginal.
Our pathwise construction then realizes the invariant state from stable input and allows us to compare queue laws as spare capacity or the input scale varies.

\subsection{Stationary balances and the regulator}

We first take stationary expectations in the regulator-cancellation identity \cref{eq:regulator-cancellation} to obtain $\E_\pi I=\beta$.
We then use the queue equation to identify the regulator occupation measure.
Define the regulator occupation measure by 
$$\eta(B):=\E_\pi\int_0^1\1_{\{Q_s\in B\}}\dd U_s\,, \quad \text{for} \quad B\in\mathcal B(\R_+).$$

\begin{proposition}[Stationary balances and regulator occupation]\label[proposition]{prop:regulator-occupation}
The invariant law satisfies $\E_\pi I=\beta$, $0<m_Q<\infty$, and
\begin{equation}\label{eq:regulator-occupation}
\eta(\dd q)=q\,\pi_Q(\dd q).
\end{equation}
Moreover, $\pi_Q((x,\infty))>0$ for every $x>0$.
For every bounded Borel or nonnegative measurable $h$,
\[
\E_\pi\int_0^t h(Q_s)\dd U_s=t\E_\pi[Qh(Q)],\qquad t>0.
\]
In particular, $\E_\pi U_t=tm_Q$ for $t\ge0$.
\end{proposition}
\begin{proof}
By \Cref{lem:stationary-moment-bounds} with $p=1$, $\E_\pi I<\infty$ and $\E_\pi Q<\infty$.
Stationarity gives $\E_\pi\int_0^t I_s\dd s=t\E_\pi I<\infty$; together with $\E|\Stable(t)|<\infty$, this justifies taking expectations in \cref{eq:regulator-cancellation}.
Using $\E\Stable(t)=0$ yields $\E_\pi I=\beta$.
The queue equation gives $U_t=Q_t-Q_0+\int_0^tQ_s\dd s$, which is integrable and has mean $t\E_\pi Q$.

For a solution started at $(i,q)$, \cref{eq:regulator-cancellation} and $I_t\ge0$ give $Q_t\ge q-i-\Stable(t)-\beta t$.
By \cref{eq:stable-negative-tail}, $\Pp_{(i,q)}(Q_t>x)>0$ for every $x,t>0$ and every initial state.
Invariance gives $\pi_Q((x,\infty))>0$ for every $x>0$, and hence $m_Q>0$.

For bounded continuous $h$, put $H(q)=\int_0^q h(v)\dd v$.
The bound $|H(q)|\le\|h\|_\infty q$ gives $\E_\pi|H(Q)|<\infty$.
The chain rule and the queue equation give
\[
H(Q_t)-H(Q_0)
=-\int_0^t Q_s h(Q_s)\dd s+\int_0^t h(Q_s)\dd U_s.
\]
Both integrals are absolutely integrable because $h$ is bounded and $\E_\pi U_t=tm_Q<\infty$.
Taking expectations gives the identity for continuous $h$.
Equality of finite measures extends it to bounded Borel $h$, and monotone convergence extends it to nonnegative measurable $h$.
Taking $t=1$ gives \cref{eq:regulator-occupation}.\hfill
\end{proof}

After normalization by $m_Q$, the regulator occupation measure is the size-biased stationary queue law: regulator increments sample $Q$ with a bias proportional to its value.
In the queueing interpretation, these increments correspond to arrivals that find no idle server.
The regulator occupation identity expresses the boundary term in stationary It\^o's formula through the queue marginal.
Here $\gen f$ is the integro-differential expression in \cref{eq:gen-jumpgen}; the boundary term is displayed separately when the oblique derivative is nonzero.
The product tests below use the growth condition
\begin{equation}\label{eq:stationary-exponential-test-class}
|h(q)|+|h'(q)|+|h''(q)|\le C\e^{\theta q},\qquad q\ge0,
\end{equation}
for some $0\le\theta<\theta_c$.

\begin{proposition}[Stationary It\^o identity with the regulator term]\label[proposition]{prop:stationary-ito}
Let $f\in\mathcal C_J^2(E)$ be bounded and satisfy
\[
\pi(|\gen f|)+\E_\pi[Q|f_i(0,Q)+f_q(0,Q)|]<\infty.
\]
Then
\begin{equation}\label{eq:stationary-ito-boundary}
\E_\pi[\gen f(I,Q)]
+\E_\pi\bigl[Q\{f_i(0,Q)+f_q(0,Q)\}\bigr]=0.
\end{equation}
This identity also holds for $f(i,q)=\e^{-si}h(q)$, where $s>0$ and $h\in C^2(\R_+)$ satisfies \cref{eq:stationary-exponential-test-class} for some $C<\infty$ and $0\le\theta<\theta_c$.
For these tests, it becomes
\begin{equation}\label{eq:stationary-product-test}
0=\Phi(s)\E_\pi[\e^{-sI}h(Q)]+s\E_\pi[I\e^{-sI}h(Q)]+\E_\pi\!\left[Q\{sh(Q)-h'(Q)\}(\e^{-sI}-1)\right].
\end{equation}
\end{proposition}

\Cref{app:stationary-ito-proof} proves the proposition.
Equation \cref{eq:stationary-ito-boundary} is a basic adjoint relationship between the stationary and boundary measures \cite{HarrisonWilliams1987,BravermanDaiMiyazawa2017}.
Here it follows by taking stationary expectations in It\^o's formula while retaining the regulator term, as in the L\'evy-input martingale methods of Kella and Whitt \cite{KellaWhitt1992}.

\subsection{Idleness weighting and absolute continuity}

Weighting each stationary state $(i,q)$ by its idle capacity $i$ gives the probability measure $(i/\beta)\pi(\dd i,\dd q)$, since $\E_\pi I=\beta$.
The next theorem identifies its queue marginal.
Let $Q_\pi\sim\pi_Q$, and let $V_{\mathrm u}$ be uniform on $(0,1)$ and independent of $Q_\pi$.

\begin{theorem}[Idleness-weighted queue distribution]\label[theorem]{thm:idleness-weighting}
For every bounded Borel or nonnegative measurable function $\varphi$,
\begin{equation}\label{eq:idleness-weighting}
\frac1\beta\E_\pi[I\varphi(Q)]=\E[\varphi(V_{\mathrm u}Q_\pi)].
\end{equation}
In particular, for every $r\ge0$,
\begin{equation}\label{eq:mixed-idleness-moments}
\E_\pi[IQ^r]=\frac{\beta}{r+1}\E_\pi Q^r,
\qquad
\operatorname{Cov}_\pi(I,Q)=-\frac\beta2m_Q<0.
\end{equation}
\end{theorem}
\begin{proof}
Take $h\in C^1(\R_+)$ with $h$ and $qh'(q)$ bounded.
The product rule, continuity and finite variation of $Q$, and $I=0$ on the support of $\dd U$ give
\[
\begin{aligned}
\dd(I_t h(Q_t))={}&h(Q_t)\dd\Stable(t)+h(Q_t)\dd U_t+\bigl[\beta h(Q_t)-I_t h(Q_t)-Q_t h(Q_t)-I_tQ_t h'(Q_t)\bigr]\dd t.
\end{aligned}
\]
The stable stochastic integral has mean zero: its integrand is bounded, the small-jump part is square-integrable, and the large jumps have finite first moment.
Because $h$ and $qh'(q)$ are bounded, the endpoint terms and $\dd t$ integrals are integrable by $\E_\pi I=\beta$ and $\E_\pi Q=m_Q<\infty$ (\Cref{prop:regulator-occupation}).
The regulator term has expected value at most $\|h\|_\infty\E_\pi U_1=\|h\|_\infty m_Q$.
Taking expectations cancels the terms involving $Qh(Q)$ and gives
\begin{equation}\label{eq:idleness-test-identity}
\E_\pi[I\{h(Q)+Qh'(Q)\}]=\beta\E_\pi h(Q).
\end{equation}
For a bounded smooth $\varphi$, choose
\[
h(q)=\frac1q\int_0^q\varphi(v)\dd v\quad(q>0),
\qquad h(0)=\varphi(0).
\]
Then $h+qh'=\varphi$ and $h(q)=\E\varphi(V_{\mathrm u}q)$, so \cref{eq:idleness-test-identity} proves \cref{eq:idleness-weighting} on this class.
Equality of finite measures extends it to bounded Borel functions, and monotone convergence extends it to nonnegative functions.
For $\varphi(q)=q^r$, both $\E_\pi Q^r$ and $\E_\pi[IQ^r]$ are finite by \Cref{lem:stationary-moment-bounds} with $p=1$, using $q^r\le C_{r,\theta_0}\e^{\theta_0q}$ for a fixed $\theta_0\in(0,\theta_c)$.
This yields \cref{eq:mixed-idleness-moments}; the case $r=1$ and $\E_\pi I=\beta$ give the covariance.\hfill
\end{proof}

The idleness-weighted queue mean is $\E[V_{\mathrm u}Q_\pi]=m_Q/2$, half the unweighted mean.
Positive idleness keeps the regulator inactive while the queue decays, so weighting observations by idle capacity favors smaller queue values.
These identities use the linear state equations, continuous reflection, stationary first moments, and the zero mean of the input integral.
In Brownian JSQ, the integer moment identities of Banerjee and Mukherjee \cite{BanerjeeMukherjee2020} determine the idleness-weighted law, since the queue has a finite positive exponential moment \cite{BanerjeeMukherjee2019}.

Conditional on $Q_\pi=v>0$, the product $V_{\mathrm u}Q_\pi$ is uniform on $(0,v)$.
We use the representation $V_{\mathrm u}Q_\pi$ to obtain an almost-everywhere positive queue density and identify conditional mean idleness.
We then prove joint absolute continuity by comparing the stationary process with its reflection-free evolution over an interval whose length tends to zero.

\begin{proposition}[Absolute continuity and conditional mean idleness]\label[proposition]{prop:joint-stationary-density}
The invariant law $\pi$ has a density with respect to two-dimensional Lebesgue measure on $E$.
Its queue marginal has a density $g_Q$ such that, for Lebesgue-almost every $q>0$, $0<g_Q(q)<\infty$ and
\begin{equation}\label{eq:conditional-idleness-mean}
\E_\pi[I\mid Q=q]
=\frac{\beta}{g_Q(q)}\int_q^\infty\frac{g_Q(v)}v\dd v.
\end{equation}
\end{proposition}
\begin{proof}
Taking $h=1$ in the product-test identity \cref{eq:stationary-product-test} gives
\begin{equation}\label{eq:idleness-laplace-balance}
\kappa_\alpha s^{\alpha-1}\E_\pi\e^{-sI}
=\beta\E_\pi\e^{-sI}+m_Q-\E_\pi[(I+Q)\e^{-sI}],\qquad s>0.
\end{equation}
The right side of \cref{eq:idleness-laplace-balance} is at most $\beta+m_Q$, so $\E_\pi\e^{-sI}\le(\beta+m_Q)/(\kappa_\alpha s^{\alpha-1})\to0$ as $s\to\infty$.
Hence $\Pp_\pi(I=0)=0$.
The proof of \Cref{prop:regulator-occupation} gives $\Pp_{(i,0)}(Q_t>0)>0$ for every $i\ge0$ and $t>0$.
By \cref{eq:queue-variation-of-constants}, $Q_t=0$ requires $Q_0=0$.
Invariance therefore rules out positive mass at $Q=0$.

Define the probability measure $\mu(B):=\beta^{-1}\E_\pi[I\1_{\{Q\in B\}}]$ on $(0,\infty)$.
Since $I>0$ almost surely, $\mu$ and $\pi_Q$ have the same null sets.
By \cref{eq:idleness-weighting}, $\mu$ is the law of $V_{\mathrm u}Q_\pi$, with density
\[
q\longmapsto\int_{(q,\infty)}\frac1v\pi_Q(\dd v),\qquad q>0.
\]
This density is at most $1/q$ and is strictly positive because $\pi_Q((q,\infty))>0$ by \Cref{prop:regulator-occupation}.
Thus $\mu$, and hence $\pi_Q$, is equivalent to Lebesgue measure on $(0,\infty)$.
Disintegration of \cref{eq:idleness-weighting} gives
\[
\E_\pi[I\mid Q=q]g_Q(q)=\beta\int_{(q,\infty)}\frac1v\pi_Q(\dd v)
\]
for Lebesgue-almost every $q>0$.
Substituting $\pi_Q(\dd v)=g_Q(v)\dd v$ and dividing by $g_Q(q)>0$ proves \cref{eq:conditional-idleness-mean}.

For joint absolute continuity, fix $t>0$ and consider the time-$t$ state without the regulator term:
\[
\left(\e^{-t}I_0+\beta(1-\e^{-t})-t\e^{-t}Q_0+Z_t,
\ \e^{-t}Q_0\right),
\qquad
Z_t:=\int_0^t\e^{-(t-s)}\dd\Stable(s).
\]
By \cref{eq:stable-integral-scaling}, $Z_t\overset d=[(1-\e^{-\alpha t})/\alpha]^{1/\alpha}\Stable(1)$, so \Cref{lem:stable-convolution-tail} gives its density for $t>0$.
Independence from the initial state follows from independence of future input increments.
The second coordinate has a density because $Q_0\sim\pi_Q$, which is absolutely continuous by the preceding argument.
Conditionally on $Q_0$, convolution with the independent density of $Z_t$ makes the first coordinate absolutely continuous, so the displayed vector has a joint density.
On $\{U_t=0\}$ it equals $Y_t$.
For any two-dimensional Lebesgue-null set $B$, stationarity gives $\pi(B)\le\Pp_\pi(U_t>0)$.
Since $I_0>0$ almost surely, right continuity gives a pathwise positive initial interval on which $I>0$ and $U=0$.
Thus $\Pp_\pi(U_t>0)\to0$ as $t\downarrow0$, proving joint absolute continuity.\hfill
\end{proof}

\subsection{Reconstruction of the joint Laplace transform}\label{sec:joint-transform}

For $s,r\ge0$, define
\[
F(s,r):=\E_\pi\e^{-sI-rQ},
\qquad
G(r):=F(0,r)=\int_0^\infty\e^{-rq}g_Q(q)\dd q.
\]
The exponential factors are bounded by one.
The drift terms involving $I$ and $Q$ can be expressed through $-F_s$ and $-F_r$, while \cref{eq:regulator-occupation} expresses the boundary term through $G'$.
This gives a first-order partial differential equation for $F$ whose inhomogeneous term is determined by the queue transform $G$.

For $s>0$ and $r\ge0$, the joint transform satisfies
\begin{equation}\label{eq:joint-transform-equation}
sF_s(s,r)+(s+r)F_r(s,r)
=\Phi(s)F(s,r)+(s+r)G'(r),
\end{equation}
where $\Phi(s)=\kappa_\alpha s^\alpha-\beta s$ and derivatives at $r=0$ are taken from the right.
Take $h(q)=\e^{-rq}$ in \cref{eq:stationary-product-test}.
Since $\E_\pi(I+Q)<\infty$, differentiation under the expectation gives $F_s=-\E_\pi[I\e^{-sI-rQ}]$, $F_r=-\E_\pi[Q\e^{-sI-rQ}]$, and $G'=-\E_\pi[Q\e^{-rQ}]$.
Substitution gives \cref{eq:joint-transform-equation}.

For a given $G$, \cref{eq:joint-transform-equation} is an ordinary differential equation along its characteristic curves.
For $r\ge0$, boundedness of $F$ determines the integration constant.
To integrate the equation, set
\[
\Xi(u):=\frac{\kappa_\alpha}{\alpha}u^\alpha-\beta u,
\qquad
\rho_{s,r}(u):=u\left(\frac rs+\log\frac us\right),\quad u\ge s>0.
\]
Then $\rho_{s,r}'(u)=1+\rho_{s,r}(u)/u$ and $\Xi'(u)=\Phi(u)/u$.

\begin{theorem}[Joint Laplace transform from the queue marginal]\label[theorem]{thm:joint-transform-reconstruction}
For $s>0$ and $r\ge0$,
\begin{equation}\label{eq:joint-transform-reconstruction}
F(s,r)=-\e^{\Xi(s)}\int_s^\infty
\e^{-\Xi(u)}\rho_{s,r}'(u)G'(\rho_{s,r}(u))\dd u.
\end{equation}
Thus the queue marginal determines the joint stationary law.
\end{theorem}
\begin{proof}
Along the characteristic, \cref{eq:joint-transform-equation} gives
\[
\frac{\dd}{\dd u}\{\e^{-\Xi(u)}F(u,\rho_{s,r}(u))\}
=\e^{-\Xi(u)}\rho_{s,r}'(u)G'(\rho_{s,r}(u)).
\]
For $u\ge s$ and $r\ge0$, $\rho_{s,r}(u)\ge0$, so $0\le F(u,\rho_{s,r}(u))\le1$.
As $u\to\infty$, $\Xi(u)\to\infty$ and $\e^{-\Xi(u)}F(u,\rho_{s,r}(u))\to0$.
Integration proves \cref{eq:joint-transform-reconstruction}.
The integral is absolutely convergent because $|G'(v)|\le m_Q$ for $v\ge0$ and $\e^{-\Xi(u)}$ decays faster than exponentially.\hfill
\end{proof}

Setting $r=0$ gives the idleness marginal transform $F(s,0)=\E_\pi\e^{-sI}$.
\Cref{app:conditional-idleness-transform} derives the conditional Laplace transform of idleness given the queue.

\subsection{Pathwise construction and parameter dependence}\label{sec:stationary-representation}

Having recovered the joint law from the queue marginal, we now construct a stationary state directly from the stable input.
Multiplying $Q_t$ by $\e^t$ and setting $x=\e^t$ removes the linear decay and leads to the recursion in \Cref{thm:stationary-representation}.
\begin{proof}[Proof of \Cref{thm:stationary-representation}]
\textit{Stationary upper bound.} Let $Y=(I,Q)$ be any two-sided stationary solution driven by a two-sided extension of $\Stable$, such that $Y_s$ is independent of the input increments after $s$ for every $s\in\R$.
Invariance of $\pi$ and pathwise uniqueness make the joint laws of the solution and input increments, initialized with $\pi$ at each finite left endpoint, consistent under restriction; their extension gives the required two-sided stationary solution.
Its marginal is $\pi$ by uniqueness of the invariant law, and the queue equation determines its regulator increments.
For $x>0$, put
\[
\overline R(x):=xQ_{\log x},
\qquad
\overline J(x):=\int_0^x\frac{\overline R(v)}v\dd v.
\]
The queue equation gives $\dd\overline R(x)=x\dd U_{\log x}$, so $\overline R$ is continuous and nondecreasing.
Stationarity and Tonelli's theorem give $\E\overline R(x)=\E\overline J(x)=xm_Q$.
In particular, $\overline R(x)\to0$ almost surely as $x\downarrow0$, so we extend $\overline R$ continuously by setting $\overline R(0)=0$; also $\overline J(x)<\infty$.

To realize the input in \Cref{thm:stationary-representation}, define $A(x):=xI_{\log x}-\overline R(x)+\overline J(x)$ for $x>0$.
The product rule gives $\dd A(x)=\beta\dd x+x\dd\Stable(\log x)$.
As a process in $x$, the stable integral $\int_{-\infty}^{\log x}\e^t\dd\Stable(t)$ has the law of a copy of $\Stable$ evaluated at the deterministic time $x^\alpha/\alpha$, by \cref{eq:stable-integral-scaling}.
The integral tends to zero in $L^1$ as $x\downarrow0$ by its stable scaling; $\E|A(x)|\le x(\beta+2m_Q)$ gives the same limit for $A(x)$.
The common $L^1$ limit at zero makes the integration constant zero.
Set $A(0)=0$ and, for $x>0$, $A(x)=\beta x+\int_{-\infty}^{\log x}\e^t\dd\Stable(t)$.
In this coupling, $A(x)=\beta x+\widehat\Stable(x^\alpha/\alpha)$, where $\widehat\Stable$ has the same law as $\Stable$.
Use this input $A$ in \cref{eq:stationary-recursion}.
The nonnegative coordinate $xI_{\log x}=A(x)+\overline R(x)-\overline J(x)$ and the support of $\dd\overline R$ imply, by the one-sided Skorokhod formula,
\[
\overline R(x)=\sup_{0\le u\le x}[\overline J(u)-A(u)]^+.
\]

\textit{Convergence of the recursion.} Monotonicity in \cref{eq:stationary-recursion} now gives $0\le R_n\le R_{n+1}\le\overline R$ and $0\le J_n\le J_{n+1}\le\overline J$.
Monotone convergence and interchange of the two suprema over $n$ and $u$ yield
\[
J_\infty(x)=\int_0^x\frac{R_\infty(v)}v\dd v,
\qquad
R_\infty(x)=\sup_{0\le u\le x}[J_\infty(u)-A(u)]^+.
\]
The integral defines a continuous $J_\infty$, including at zero, by domination by $\overline R(v)/v$.
The path $J_\infty-A$ has no positive jumps, so its running supremum is continuous.
The same argument applies at each finite step.
Dini's theorem therefore gives uniform convergence on compact intervals.

\textit{Identification of the invariant law.} Strict stability gives $c^{-1}A(c\,\cdot)\overset d=A(\cdot)$ for every $c>0$.
The recursion commutes with this scaling, so the process $Y^*=(I^*,Q^*)$ defined by
\[
Q_t^*:=\e^{-t}R_\infty(\e^t),
\qquad
I_t^*:=\e^{-t}\{A(\e^t)+R_\infty(\e^t)-J_\infty(\e^t)\}
\]
is stationary under time shifts, as in the Lamperti transformation \cite{Lamperti1962}.
For $t\ge0$, set $U_t^*:=\int_1^{\e^t}v^{-1}\dd R_\infty(v)$.
The identity $\dd A(v)=\beta\dd v+v\dd\Stable(\log v)$ and the product rule give
\[
\dd Q_t^*=-Q_t^*\dd t+\dd U_t^*,
\qquad
\dd I_t^*=(\beta-I_t^*-Q_t^*)\dd t+\dd\Stable(t)+\dd U_t^*.
\]
The initial state depends only on input increments on $(-\infty,0]$ and is therefore independent of future increments of $\Stable$.
The running-supremum condition makes $U^*$ continuous and allows it to increase only when $I^*=0$.
Thus this is a stationary solution of the reflected SDE.
Uniqueness of the invariant law proves \cref{eq:stationary-law-representation}.
The iterates and limits are measurable input functionals, so these almost-sure properties hold for every copy of $\widehat\Stable$.

\textit{Recovery from the input.} For each $x>0$, $R_\infty(x)\le\overline R(x)$ and stationarity gives $\E R_\infty(x)=xm_Q=\E\overline R(x)$, so equality holds almost surely.
Equality at positive rational $x$ and continuity give $R_\infty=\overline R$ and $J_\infty=\overline J$ on $[0,\infty)$.
Thus $xQ_{\log x}=R_\infty(x)$ and $xI_{\log x}=A(x)+R_\infty(x)-J_\infty(x)$ for all $x>0$.
The recursion at $x$ uses only input increments up to time $\log x$, so $Y_0$ is measurable with respect to the input on $(-\infty,0]$.
Any two stationary solutions satisfying the stated independence condition and driven by the same input are therefore indistinguishable. \hfill
\end{proof}

The recursion also shows how spare capacity and input scale affect the stationary queue.
\begin{proposition}[Monotonicity and convexity]\label[proposition]{prop:queue-parameter-dependence}
For fixed $\alpha,c_{\mathrm A}$, the stationary queue is stochastically nonincreasing in $\beta$.
For every nondecreasing convex $\varphi:\R_+\to\R_+$, $\E_\pi[\varphi(Q)]$ is also nonincreasing and convex in $\beta$ on its domain of finiteness.
For fixed $\alpha,\beta$, the stationary queue is stochastically nondecreasing in $c_{\mathrm A}$.
\end{proposition}
\begin{proof}
Fix $\alpha$ and let $s_\alpha:=\kappa_\alpha^{1/\alpha}$ be the noise-amplitude scale.
Let $\bar Q_b$ denote the stationary queue with spare-capacity parameter $b>0$ and normalized input Laplace exponent $\theta^\alpha$, and write $\bar m_\alpha(b):=\E\bar Q_b$.
Dividing both state coordinates and the regulator by $c>0$ changes the parameters to $(\beta/c,c_{\mathrm A}c^{-\alpha})$.
Taking $c=s_\alpha$ shows that, for $(I,Q)\sim\pi$, the vector $(I/s_\alpha,Q/s_\alpha)$ has the invariant law with spare-capacity parameter $b=\beta/s_\alpha$ and input Laplace exponent $\theta^\alpha$.
Uniqueness of the invariant law then gives
\[
Q\overset d=s_\alpha\bar Q_b,
\qquad
m_Q(\alpha,\beta,c_{\mathrm A})=s_\alpha\bar m_\alpha(b),
\qquad
\theta_c=s_\alpha^{-1}b^{1/(\alpha-1)}.
\]
Run \cref{eq:stationary-recursion} with one common normalized stable path and spare-capacity parameter $b$, denoting the iterates by $R_n^b,J_n^b$.
Inductively, for each $x$, both iterates are nonincreasing and convex in $b$: integration against $\dd v/v$, subtraction of the linear term $bu$, positive parts, and suprema preserve these properties.
On the probability-one event where the recursion is finite for every positive rational $b$, monotonicity gives finiteness for all $b>0$.
The inequalities pass to the limits, so $\bar Q_b:=R_\infty^b(1)$ couples the normalized stationary queue laws in decreasing order by \Cref{thm:stationary-representation}.
Composition with a nondecreasing convex $\varphi$ preserves nonincrease and convexity.
Taking expectations and applying the scaling identity with $\varphi(s_\alpha\,\cdot)$ proves the stated dependence on $\beta$ for fixed $\alpha,c_{\mathrm A}$.
For $0<s<s'$, the same coupling gives $s\bar Q_{\beta/s}\le s\bar Q_{\beta/s'}\le s'\bar Q_{\beta/s'}$ almost surely.
Since $s_\alpha$ increases with $c_{\mathrm A}$, the scaling identity proves the noise-scale ordering. \hfill
\end{proof}
For fixed $\alpha$, the scaling relation $m_Q=s_\alpha\bar m_\alpha(\beta/s_\alpha)$, where $s_\alpha=\kappa_\alpha^{1/\alpha}$ and $\bar m_\alpha$ is the mean queue under normalized input, expresses $m_Q$ as the perspective of a convex function.
Hence $m_Q$ is jointly convex in $(\beta,s_\alpha)$ and positively homogeneous of degree one.

\section{Tail and boundary asymptotics}\label{sec:tails}

We study the stationary regimes of large idleness, large queues, and small idleness, and prove the tail, moment, and boundary results stated in \Cref{thm:marginal-tails,thm:invariant-moments,thm:boundary-asymptotics}.
The stationary convolution gives the power-law idleness tail, a regenerative decomposition gives the sharp queue tail, and the regulator identity gives the boundary limit.
These asymptotics yield the exact moment domains and the corresponding conditional laws.

\subsection{The idleness tail}\label{sec:idleness-tail}

We work with the two-sided stationary solution constructed in \Cref{sec:stationary-representation}.
Set
\[
Z:=\int_{-\infty}^0\e^s\dd\Stable(s),
\qquad
J:=\int_{-\infty}^0\e^s Q_s\dd s=\overline J(1).
\]
Using $A(1)=I_0-Q_0+J=\beta+Z$ and the stable scaling identity \cref{eq:stable-integral-scaling}, we obtain
\begin{equation}\label{eq:stationary-convolution}
I_0=Z+\beta+Q_0-J,
\qquad
I_0-Q_0=Z+\beta-J,
\qquad
Z\overset d=\alpha^{-1/\alpha}\Stable(1).
\end{equation}
For $0\le\theta<\theta_c$, \Cref{lem:stationary-moment-bounds} gives $\E_\pi\e^{\theta Q}<\infty$.
Jensen's inequality for the probability measure $\e^s\dd s$ on $(-\infty,0]$ and stationarity give
\begin{equation}\label{eq:stationary-queue-convolution-bound}
\E_\pi\e^{\theta J}
\le\int_{-\infty}^0\e^s\E_\pi\e^{\theta Q_s}\dd s
=\E_\pi\e^{\theta Q}<\infty.
\end{equation}

The random variable $Z$ in \cref{eq:stationary-convolution} has the classical stationary stable OU law \cite{Sato1999}; its tail follows from regular variation of infinitely divisible laws \cite{HultLindskog2006}.
Both remainders $H=\beta+Q_0-J$ and $H=\beta-J$ satisfy $\E_\pi\e^{\lambda_0|H|}<\infty$ for some $\lambda_0>0$.
Upper and lower bounds transfer $Z$'s tail equivalent to $I_0$ and $I_0-Q_0$, without requiring independence.

\begin{proof}[Proof of \cref{eq:idleness-tail} in \Cref{thm:marginal-tails}]
By \Cref{lem:stable-convolution-tail} and \cref{eq:stationary-convolution}, $\lim_{x\to\infty}x^\alpha\Pp(Z>x)=\frac{\cA}{\alpha^2}$.
Choose $\lambda_0>0$ with $2\lambda_0<\theta_c$.
The moment bound \cref{eq:stationary-queue-convolution-bound} and Cauchy--Schwarz inequality give
\[
\E_\pi\e^{\lambda_0|\beta+Q_0-J|}
\le\e^{\lambda_0\beta}
 (\E_\pi\e^{2\lambda_0 Q_0})^{1/2}
 (\E_\pi\e^{2\lambda_0 J})^{1/2}<\infty.
\]
The same $\lambda_0$ gives $\E_\pi\e^{\lambda_0|\beta-J|}<\infty$.
For either remainder $H=\beta+Q_0-J$ or $H=\beta-J$, bound $Z+H$ as follows.
For $0<\varepsilon<1$,
\[
\Pp(Z>(1+\varepsilon)x)-\Pp(|H|>\varepsilon x)
\le\Pp(Z+H>x)
\le\Pp(Z>(1-\varepsilon)x)+\Pp(|H|>\varepsilon x).
\]
The remainder probability is exponentially small.
Multiply by $x^\alpha$, let $x\to\infty$, and then let $\varepsilon\downarrow0$.
Both limits are $\cA/\alpha^2$. \hfill
\end{proof}
Integrating the jump-tail coefficient over the exponential decay of past jumps gives $\cA/\alpha^2=\int_0^\infty(\cA/\alpha)\e^{-\alpha s}\dd s$.
For a fixed input scale, this constant is independent of $\beta$: the constant drift and the queue coordinate enter only through exponentially integrable remainders, while unit-rate mean reversion discounts the large jumps.

\subsection{Sharp queue-tail asymptotics}\label{sec:queue-tail-analysis}

\subsubsection{Regenerative reduction}
For $b>0$, let
\[
\tau_b^\uparrow:=\inf\{t\ge0:Q_t\ge b\},\qquad
\tau_b^\downarrow:=\inf\{t\ge0:Q_t\le b\}.
\]
At a first upcrossing of a queue level $b$ from below, the state is $(0,b)$.
To prove this claim, note that $Q$ is continuous and reaches $b$ from below, so the hitting time is predictable.
The L\'evy input has no jump at a predictable stopping time.
If the left limit of idleness were positive, idleness would remain positive just before the hit, the regulator would be inactive, and $Q$ would decrease, a contradiction.
Thus both $I_{\tau_b^\uparrow-}$ and $I_{\tau_b^\uparrow}$ are zero.

Fix $\theta_0\in(0,\theta_c)$ and a corresponding threshold $q_{\mathrm L}$ from \Cref{lem:killed-decay} in \Cref{app:queue-downcrossing}.
The lemma bounds the tail of the downcrossing time and supplies uniform integrable bounds for the occupation probabilities used below.
For $a>\ell>0$ and a solution started at $(0,a)$,  define the cycle time
\[
\mathsf T_{\ell,a}:=\inf\{t>\tau_\ell^\downarrow:Q_t\ge a\}.
\]
For $x>b>\ell$, define $\mathsf p_\ell(b,x):=\Pp_{(0,b)}(\tau_x^\uparrow<\tau_\ell^\downarrow)$, with $\mathsf p_\ell(b,b):=1$, and let
\[
\mathcal O_\ell^Q(x):=\E_{(0,x)}\int_0^{\tau_\ell^\downarrow}\1_{\{Q_t>x\}}\dd t,
\qquad
\mathcal O_\ell^{Q-I}(x):=\E_{(0,x)}\int_0^{\tau_\ell^\downarrow}\1_{\{Q_t-I_t>x\}}\dd t.
\]
\Cref{fig:regenerative-cycle} illustrates a cycle with $\tau_x^\uparrow<\tau_\ell^\downarrow$.

\begin{figure}[hptb]
\centering
\includegraphics[width=\linewidth]{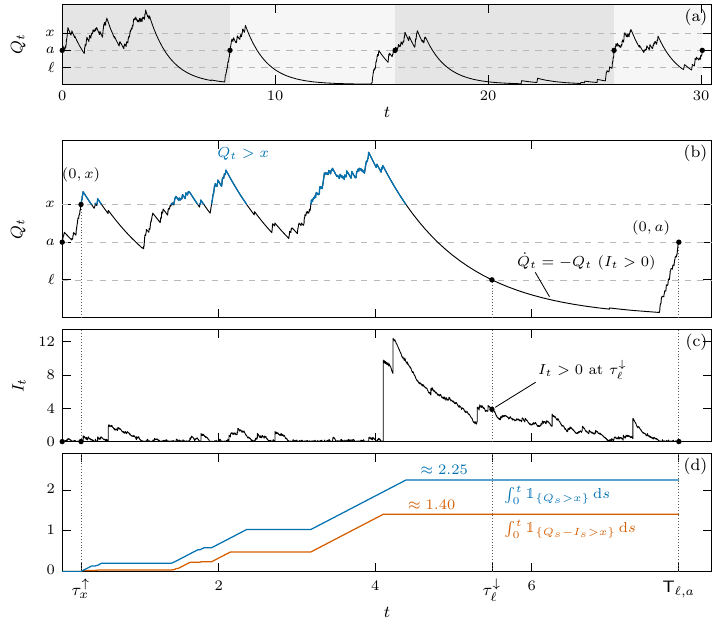}
\caption{Numerical illustration of a regenerative cycle in \Cref{lem:queue-regeneration} for $\alpha=3/2$, $\beta=\cA=1$, and $(\ell,a,x)=(1,2,3)$.
(a) Four independently simulated cycles, selected for display and concatenated at $(0,a)$; shading alternates between cycles, and dots mark the regeneration states $(0,a)$.
(b, c) Queue and idleness along the first cycle; dotted lines mark $\tau_x^\uparrow$, $\tau_\ell^\downarrow$, and $\mathsf T_{\ell,a}$, and the queue is blue where $Q_t>x$.
The cycle starts at $(0,a)$, reaches $(0,x)$ at $\tau_x^\uparrow$, and has positive idleness at $\tau_\ell^\downarrow$, so it closes only at $\mathsf T_{\ell,a}$, where the state is again $(0,a)$.
(d) The occupation times $\int_0^t\1_{\{Q_s>x\}}\dd s$ and $\int_0^t\1_{\{Q_s-I_s>x\}}\dd s$ increase only on $[\tau_x^\uparrow,\tau_\ell^\downarrow]$; their final values are the occupation integrals in the definitions of $\mathcal O_\ell^Q(x)$ and $\mathcal O_\ell^{Q-I}(x)$, evaluated on the path after $\tau_x^\uparrow$.}
\label{fig:regenerative-cycle}
\end{figure}

\begin{lemma}[Regenerative tail representation]\label[lemma]{lem:queue-regeneration}
Let $\ell>q_{\mathrm L}$ satisfy $\Pp_\pi(Q<\ell)>0$, and let $a>\ell$.
Successive returns to $(0,a)$, each after a visit to the lower queue level $\ell$, define independent and identically distributed cycles, with $0<\E_{(0,a)}\mathsf T_{\ell,a}<\infty$.
For $x>a$,
\begin{equation}\label{eq:queue-regenerative-representation}
\Pp_\pi(Q>x)=\frac{\mathsf p_\ell(a,x)\mathcal O_\ell^Q(x)}{\E_{(0,a)}\mathsf T_{\ell,a}},
\qquad
\Pp_\pi(Q-I>x)=\frac{\mathsf p_\ell(a,x)\mathcal O_\ell^{Q-I}(x)}{\E_{(0,a)}\mathsf T_{\ell,a}}.
\end{equation}
\end{lemma}
\begin{proof}
By \Cref{prop:regulator-occupation}, $\Pp_\pi(Q>a)>0$ for every $a>0$.
Positive Harris recurrence and the ergodic theorem ensure that $Q_t$ visits levels below $\ell$ and above $a$ from every initial state.
Cycle times are finite almost surely; the upcrossing state and strong Markov property make cycles independent and identically distributed.

To check the mean cycle length, let $L_n$ be the successive cycle lengths and $O_n$ the time spent with $Q_t>a$ during the $n$th cycle.
There is no such occupation during the return from $\ell$ to $a$, so \cref{eq:killed-occupation-bound} gives $\E O_1\le C/\delta<\infty$ for the constants $C,\delta$ in that bound.
If $\E L_1=\infty$, applying the strong law to $L_j\wedge m$ and then letting $m\to\infty$ gives $n^{-1}\sum_{j=1}^nL_j\to\infty$.
The strong law for the i.i.d. cycle occupation times $O_j$ would then imply $\sum_{j=1}^nO_j/\sum_{j=1}^nL_j\to0$, contrary to the Harris ergodic theorem, which makes the same ratio converge to $\Pp_\pi(Q>a)>0$.
Hence the mean cycle length is finite.
It is positive because \cref{eq:queue-variation-of-constants} implies $L_1\ge\log(a/\ell)$.

The renewal-reward theorem expresses each stationary probability as its mean cycle occupation divided by $\E L_1$.
Since $Q-I\le Q$, neither event can occur before the first upcrossing of $x$, and neither occurs on the return from $\ell$ to $a$.
The strong Markov property at that upcrossing, whose state is $(0,x)$, factors the mean reward as in \cref{eq:queue-regenerative-representation}. \hfill
\end{proof}

To evaluate the occupation factors, let $X_t:=-\Stable(t)-\beta t$ be the unreflected net-input process.
It is spectrally negative, with $\log\E\e^{\theta X_t}=t\Phi(\theta)$ for $\theta\ge0$ and positive Lundberg root $\theta_c$ \cite{Kyprianou2014}.
Write $\mu_c:=\Phi'(\theta_c)=(\alpha-1)\beta>0$.

\begin{lemma}[Occupation limit after high-level upcrossing]\label[lemma]{lem:high-level-occupation}
For each fixed $\ell>q_{\mathrm L}$,
\begin{equation}\label{eq:high-level-occupation-limit}
\lim_{x\to\infty}\mathcal O_\ell^Q(x)
=\lim_{x\to\infty}\mathcal O_\ell^{Q-I}(x)
=\frac1{\theta_c\mu_c}.
\end{equation}
\end{lemma}
\begin{proof}
Couple all initial states $(0,x)$ with one stable path.
On $[0,T]$, \cref{eq:queue-variation-of-constants} gives $Q_t^{(0,x)}\ge x\e^{-T}$.
The one-dimensional Skorokhod representation therefore implies
\[
0\le I_t^{(0,x)}
\le\sup_{0\le u\le t}\{\Stable(t)-\Stable(u)-(x\e^{-T}-\beta)(t-u)\}.
\]
For large $x$, the right side is bounded by $2\sup_{u\le T}|\Stable(u)|$ and tends to zero at every continuity point of $\Stable$.
Consequently $I_t^{(0,x)}\to0$ almost surely at each fixed deterministic $t>0$, and $\int_0^T I_u^{(0,x)}\dd u\to0$ by dominated convergence.
Regulator cancellation gives
\[
Q_t^{(0,x)}-x=X_t+I_t^{(0,x)}+\int_0^tI_u^{(0,x)}\dd u.
\]
Also $\tau_\ell^\downarrow\ge\log(x/\ell)$, so the downcrossing occurs after any fixed $T$ for all sufficiently large $x$.
The density in \Cref{lem:stable-convolution-tail} gives $\Pp(X_t=0)=0$.
For each fixed $t>0$, both $\Pp_{(0,x)}(Q_t>x,\ t<\tau_\ell^\downarrow)$ and $\Pp_{(0,x)}(Q_t-I_t>x,\ t<\tau_\ell^\downarrow)$ converge to $\Pp(X_t>0)$ as $x\to\infty$.
We use \cref{eq:killed-occupation-bound} to dominate these probabilities by an integrable function of $t$, so both occupation factors converge to $\int_0^\infty\Pp(X_t>0)\dd t$.

On the positive half-line, the zero-resolvent measure gives expected occupation times:
\[
\int_0^\infty\Pp(X_t\in\dd v)\dd t
=\frac{\e^{-\theta_c v}}{\mu_c}\dd v,\qquad v>0,
\]
by Kuznetsov et~al. (\cite{KuznetsovKyprianouRivero2012}, Theorem~2.7\textup{(iv)}).
The inverse of $\Phi$ restricted to $[\theta_c,\infty)$ has derivative $1/\Phi'(\theta_c)=1/\mu_c$ at zero.
Integrating over $v>0$ gives $\int_0^\infty\Pp(X_t>0)\dd t=1/(\theta_c\mu_c)$, proving \cref{eq:high-level-occupation-limit}. \hfill
\end{proof}

For fixed levels $\ell,a$ as in \Cref{lem:queue-regeneration}, \cref{eq:queue-regenerative-representation,eq:high-level-occupation-limit} give
\[
\Pp_\pi(Q>x)\sim\Pp_\pi(Q-I>x)
\sim\frac{\mathsf p_\ell(a,x)}{\theta_c\mu_c\,\E_{(0,a)}\mathsf T_{\ell,a}},\qquad x\to\infty.
\]
It remains to determine the sharp asymptotic of the hitting probability.

\subsubsection{Sharp high-level hitting probabilities}
For $q>0$, define
\begin{equation}\label{eq:critical-hitting-weight}
w_c(q):=\e^{\theta_cq}q^{-\alpha/2}.
\end{equation}
The following barriers compare high-level hitting probabilities with ratios of $w_c$.

\begin{lemma}[Barriers for high-level hitting]\label[lemma]{lem:critical-barriers}
There are constants $q_{\mathrm b}>0$ and $\zeta_-,\zeta_+\in\R$ and functions $f_-,f_+\in C^2(\R_+\times(0,\infty))$ such that, for $q\ge q_{\mathrm b}$,
\[
\gen f_-\ge0,\qquad \gen f_+\le0<f_+,
\qquad (f_\pm)_i(0,q)+(f_\pm)_q(0,q)=0,
\]
\begin{equation}\label{eq:critical-barrier-boundary}
f_\pm(0,q)=w_c(q)(1+\zeta_\pm/q)>0.
\end{equation}
Both functions are bounded on every strip $\R_+\times[q_{\mathrm b},x]$.
Their compensated jump remainders satisfy the integrability condition defining $\mathcal C_J^2(E)$ on every compact subset of each strip.
\end{lemma}
\begin{proof}
We first match the generator cancellations to the oblique derivative condition, and then enforce the generator signs uniformly in idleness.
Set $h_c(i,q):=\e^{\theta_c(q-i)}$ and $\sigma_c^2:=\Phi''(\theta_c)$.
The function $h_c$ has zero oblique derivative and satisfies $\gen h_c=\theta_ci\,h_c$.
After factoring out $h_c$, the drift term $-q\partial_i$ suggests corrections in successive powers of $q^{-1}$.
For a trial exponent $\gamma$ and $\zeta\in\R$, define the polynomials $c_1(\cdot;\zeta)$ and $c_2(\cdot;\zeta)$ by $c_2(0;\zeta)=0$ and
\begin{equation}\label{eq:critical-polynomial-recursion}
\begin{split}
c_1(i;\zeta)&=\zeta+\gamma i+\frac{\theta_c}{2}i^2,\\
c_2'(i;\zeta)&=\frac{\theta_c\sigma_c^2}{2}-(\mu_c+i)(\gamma+\theta_ci)+(\gamma+1+\theta_ci)c_1(i;\zeta).
\end{split}
\end{equation}
The profile
\begin{equation}\label{eq:critical-profile}
b_\zeta(i,q):=h_c(i,q)q^{-\gamma}
\left[1+\frac{c_1(i;\zeta)}q+\frac{c_2(i;\zeta)}{q^2}\right]
\end{equation}
then cancels the terms of orders $q^{-\gamma}$ and $q^{-\gamma-1}$ in $h_c^{-1}\gen b_\zeta$; \Cref{app:critical-barriers} gives the exact generator identities.
At the reflecting face, direct differentiation gives
\[
(\partial_i+\partial_q)b_\zeta(0,q)
=\e^{\theta_cq}q^{-\gamma-2}
\left(\frac{\theta_c\sigma_c^2}{2}-\mu_c\gamma\right).
\]
Thus the oblique derivative condition determines
\begin{equation}\label{eq:critical-tail-power}
\gamma=\frac{\theta_c\sigma_c^2}{2\mu_c}=\frac\alpha2.
\end{equation}
Fix this value of $\gamma$ and recall $\psi_0(i)=i^2/(1+i)^2$ from \cref{eq:polynomial-idleness-correction}.
For constants $\zeta_-,\zeta_+\in\R$ and $C_{\mathrm b}>0$ to be chosen, define
\begin{equation}\label{eq:critical-barrier-construction}
f_\pm(i,q):=b_{\zeta_\pm}(i,q)
\pm C_{\mathrm b}w_c(q)q^{-3}\psi_0(i).
\end{equation}
The correction and its oblique derivative vanish at $i=0$, preserving the boundary condition and the values $f_\pm(0,q)=w_c(q)(1+\zeta_\pm/q)$.
\Cref{app:critical-barriers} chooses $\zeta_\pm$ to give opposite signs to the leading generator terms at the face, then chooses $C_{\mathrm b}$ and $q_{\mathrm b}$ to extend these signs to all $i\ge0$ and $q\ge q_{\mathrm b}$ and verify positivity of $f_+$, boundedness, and jump integrability as stated. \hfill
\end{proof}
We now stop $f_-$ and $f_+$ at the first exit from $\ell<Q<x$ to estimate the probability of reaching $x$ before $\ell$.

\begin{lemma}[Sharp high-level hitting probabilities]\label[lemma]{lem:sharp-queue-hitting}
Fix $\ell>\max\{q_{\mathrm L},q_{\mathrm b}\}$ with $\Pp_\pi(Q<\ell)>0$.
For every sufficiently large fixed $a>\ell$, there is $\mathsf k_\ell(a)\in(0,\infty)$ such that
\begin{equation}\label{eq:sharp-queue-hitting}
\mathsf p_\ell(a,x)
=\mathsf k_\ell(a)x^{\alpha/2}\e^{-\theta_cx}[1+O(x^{-1})],\qquad x\to\infty.
\end{equation}
\end{lemma}
\begin{proof}
We first prove the uniform estimate
\begin{equation}\label{eq:uniform-hitting-ratio}
\mathsf p_\ell(b,x)\frac{w_c(x)}{w_c(b)}=1+O(b^{-1}),\qquad b\to\infty,\quad x\ge b.
\end{equation}
For $\ell<b<x$, set $\tau=\tau_x^\uparrow\wedge\tau_\ell^\downarrow$, which is finite by \cref{eq:killed-exit-bound}.
Apply the stopped It\^o argument from the proof of \Cref{prop:generator-criterion} to $f_\pm$ up to $t\wedge\tau\wedge\tau_m$.
Since $Q$ is continuous and jumps change only $I$, all queue values lie in $[\ell,x]$.
Thus the smoothness, oblique derivative condition, and jump-integrability condition in \Cref{lem:critical-barriers} suffice.
The regulator term vanishes, and the generator signs give the corresponding stopped sub- and supermartingale inequalities.
Set $B_\ell:=\sup_{i\ge0}f_-^+(i,\ell)<\infty$.
At the upper exit the state is $(0,x)$; at the lower exit, $f_+\ge0$ and $f_-\le B_\ell$.
Bounded convergence in the stopping inequalities gives
\begin{equation*}
\frac{f_-(0,b)-B_\ell}{f_-(0,x)}
\le\mathsf p_\ell(b,x)
\le\frac{f_+(0,b)}{f_+(0,x)}.
\end{equation*}
Substitute \cref{eq:critical-barrier-boundary}.
The factors $1+\zeta_\pm/x$ differ from one by $O(b^{-1})$ uniformly for $x\ge b$, and $B_\ell/w_c(b)=o(b^{-1})$.
This proves \cref{eq:uniform-hitting-ratio}; $x=b$ is immediate.

Choose $a$ sufficiently large that $f_-(0,a)>B_\ell$.
These upper and lower bounds show that $w_c(x)\mathsf p_\ell(a,x)$ is bounded above and away from zero for $x\ge a$.
The strong Markov property at an intermediate upcrossing gives $\mathsf p_\ell(a,x)=\mathsf p_\ell(a,b)\mathsf p_\ell(b,x)$ for $a\le b\le x$.
Together with \cref{eq:uniform-hitting-ratio}, this yields
\[
\frac{w_c(x)\mathsf p_\ell(a,x)}{w_c(b)\mathsf p_\ell(a,b)}
=1+O(b^{-1}),\qquad b\to\infty,\quad x\ge b\ge a.
\]
Since $w_c(b)\mathsf p_\ell(a,b)$ is bounded above, the difference of the normalized hitting probabilities is $O(b^{-1})$, uniformly for $x\ge b$.
Thus $w_c(x)\mathsf p_\ell(a,x)$ forms a Cauchy family as $x\to\infty$.
Its limit $\mathsf k_\ell(a)$ is finite and strictly positive by the upper and lower bounds established above.
Letting $x\to\infty$ in the difference bound gives an $O(b^{-1})$ error, proving \cref{eq:sharp-queue-hitting}.\hfill
\end{proof}

\subsubsection{Stationary tails and consequences}
\begin{proof}[Proof of \cref{eq:queue-sharp-tail} in \Cref{thm:marginal-tails}]
Choose $\ell$ and $a$ as in \Cref{lem:sharp-queue-hitting}.
Substitute \cref{eq:high-level-occupation-limit,eq:sharp-queue-hitting} into the two cases of \cref{eq:queue-regenerative-representation}.
Both right tails are asymptotic to $C_\pi x^{\alpha/2}\e^{-\theta_cx}$, with
\begin{equation}\label{eq:queue-tail-constant}
C_\pi:=\frac{\mathsf k_\ell(a)}
{\theta_c\mu_c\,\E_{(0,a)}\mathsf T_{\ell,a}}\in(0,\infty).
\end{equation}
The two occupation factors have the same limit, so the constants agree.
This representation is independent of the sufficiently large auxiliary levels, since it equals the stationary tail constant. \hfill
\end{proof}

The idleness and queue tails establish divergence at the critical moment parameters, while \Cref{lem:stationary-moment-bounds} supplies finiteness below them.
\begin{proof}[Proof of \Cref{thm:invariant-moments}]
For $p<\alpha$ and $\theta<\theta_c$, choose $p<\bar p<\alpha$ and $\theta<\bar\theta<\theta_c$.
The pointwise bound $(1+I)^p\e^{\theta Q}\le(1+I)^{\bar p}\e^{\bar\theta Q}$ and \Cref{lem:stationary-moment-bounds} give finiteness.
If $p\ge\alpha$, the idleness tail in \cref{eq:idleness-tail} and tail integration give $\E_\pi I^p=\infty$.
The queue tail in \cref{eq:queue-sharp-tail} gives
\[
\E_\pi\e^{\theta_cQ}
=1+\theta_c\int_0^\infty\e^{\theta_cx}\Pp_\pi(Q>x)\dd x
=\infty,
\]
because the integrand is asymptotic to $C_\pi x^{\alpha/2}$.
Larger exponential orders diverge by monotonicity.
The joint weight dominates each marginal factor, proving necessity in \cref{eq:exact-invariant-moment-domain}.\hfill
\end{proof}

The idleness-weighting identity determines mean idleness conditional on a large queue, while the product-test identity identifies its typical scale.
Let $H_\alpha$ have the Mittag--Leffler distribution of parameter $\alpha-1$ \cite{Pillai1990}, characterized by
\[
\E\e^{-uH_\alpha}=\frac{1}{1+u^{\alpha-1}},\qquad u\ge0.
\]
For $x>0$, put $a_x:=(x/\kappa_\alpha)^{1/(\alpha-1)}$.
Write $\operatorname{Exp}(\theta_c)$ for the exponential law of rate $\theta_c$.

\begin{corollary}[Conditional idleness and queue excess]\label[corollary]{cor:large-queue-idleness}
As $x\to\infty$,
\begin{equation}\label{eq:large-queue-conditioning}
\begin{split}
x\E_\pi[I\mid Q>x] &\longrightarrow\frac\beta{\theta_c},\qquad \\
\law_\pi(a_xI,Q-x\mid Q>x)&\Longrightarrow\law(H_\alpha)\otimes\operatorname{Exp}(\theta_c).
\end{split}
\end{equation}
\end{corollary}
\begin{proof}
Equation \cref{eq:idleness-weighting} and tail integration give
\begin{equation}\label{eq:idleness-weighted-queue-tail}
\E_\pi[I;Q>x]
=\beta\E_\pi(1-x/Q)^+
=\beta\int_x^\infty\frac{x}{v^2}\Pp_\pi(Q>v)\dd v
\sim\frac{\beta C_\pi}{\theta_c}x^{\alpha/2-1}\e^{-\theta_cx}.
\end{equation}
For the last step, put $v=x+z$ and apply \cref{eq:queue-sharp-tail}; the normalized integrand is dominated by a constant times $\e^{-\theta_cz}$ since $\alpha/2-2<0$.
Division by \cref{eq:queue-sharp-tail} proves the mean limit.

Write $\overline F_Q(x):=\Pp_\pi(Q>x)$.
The tail equivalent \cref{eq:queue-sharp-tail} gives $\overline F_Q(x+z)/\overline F_Q(x)\to\e^{-\theta_cz}$ for each $z\ge0$, so $Q-x$ conditional on $Q>x$ converges to $\operatorname{Exp}(\theta_c)$.
The ratio is bounded by $C(1+z)^{\alpha/2}\e^{-\theta_cz}$ for all sufficiently large $x$ and all $z\ge0$; tail integration therefore gives $\E_\pi[Q-x\mid Q>x]=O(1)$.
Fix $u>0$ and $r\ge0$, and set $s_x:=ua_x$ and $\varepsilon_x:=a_x^{-1/2}$.
Choose a smooth $\chi:\R\to[0,1]$ equal to zero on $(-\infty,0]$ and one on $[1,\infty)$, and let
\[
h_x(q):=\chi\left(\frac{q-x}{\varepsilon_x}\right)\e^{-r(q-x)}.
\]
For each $x$, $h_x$ and its first two derivatives are bounded, so \cref{eq:stationary-product-test} applies with $s=s_x$ and $h=h_x$.
Since $\overline F_Q(x+\varepsilon_x)/\overline F_Q(x)\to1$, replacing $h_x(Q)$ by $\e^{-r(Q-x)}\1_{\{Q>x\}}$ in normalized expectations incurs an $o(1)$ error.
Replacing $Q/x$ by one incurs an $O(x^{-1})$ error by the conditional excess-mean bound.
Moreover, $i\e^{-s_xi}\le(\e s_x)^{-1}$ and $|h_x'(q)|\le(C/\varepsilon_x+r)\1_{\{q>x\}}$ give
\[
\frac{\E_\pi[I\e^{-s_xI}h_x(Q)]}{x\overline F_Q(x)}
\le\frac{1}{\e s_xx},
\qquad
\frac{\E_\pi[Q|h_x'(Q)|(1-\e^{-s_xI})]}{s_xx\overline F_Q(x)}
=O\bigl((s_x\varepsilon_x)^{-1}\bigr)=o(1).
\]
Divide \cref{eq:stationary-product-test} by $s_xx\overline F_Q(x)$.
Since $\kappa_\alpha s_x^{\alpha-1}/x=u^{\alpha-1}$, these estimates yield
\[
(1+u^{\alpha-1})\E_\pi[\e^{-ua_xI-r(Q-x)}\mid Q>x]
=\E_\pi[\e^{-r(Q-x)}\mid Q>x]+o(1)\longrightarrow\frac{\theta_c}{\theta_c+r}.
\]
Taking $r=0$ gives the Mittag--Leffler marginal by the continuity theorem for Laplace transforms.
Together with the queue-excess limit, this gives tightness of the pair; the joint Laplace transform identifies every subsequential limit as the product law in \cref{eq:large-queue-conditioning}.\hfill
\end{proof}

The Mittag--Leffler limit has infinite mean, as follows from its Laplace transform.
Idleness conditional on $Q>x$ typically has scale $x^{-1/(\alpha-1)}$, smaller than its conditional-mean scale $x^{-1}$.
Indeed, $a_x/x\to\infty$, so $xI\to0$ in conditional probability, while $\E_\pi[xI\mid Q>x]\to\beta/\theta_c>0$.
The family is therefore not uniformly integrable: rarer, larger idleness values determine the conditional mean.

\Cref{app:critical-queue-moments} gives critical queue weights and near-critical transform asymptotics.

\subsection{Small-idleness asymptotics}\label{sec:boundary-asymptotics}

As $s\to\infty$, the bounded test function $\e^{-si}$ suppresses contributions from idleness bounded away from zero.
The jump operator contributes the coefficient $\kappa_\alpha s^\alpha$, whereas the oblique derivative at $i=0$ is $-s$; their stationary balance produces the exponent $\alpha-1$.
Retaining a test function of $Q$ identifies the limiting queue distribution conditional on small idleness.

\begin{proof}[Proof of \Cref{thm:boundary-asymptotics}]
Let $h\ge0$ be smooth with compact support, or let $h(q)=\e^{\theta q}$ with $0\le\theta<\theta_c$.
Divide \cref{eq:stationary-product-test} by $s$ and rearrange:
\begin{equation}\label{eq:weighted-boundary-laplace-balance}
\begin{split}
\kappa_\alpha s^{\alpha-1}\E_\pi[\e^{-sI}h(Q)]
={}&\beta\E_\pi[\e^{-sI}h(Q)]
-\E_\pi[(I+Q)\e^{-sI}h(Q)]\\
&+s^{-1}\E_\pi[Q\e^{-sI}h'(Q)]
+\eta(h)-s^{-1}\eta(h').
\end{split}
\end{equation}
For compactly supported $h$, boundedness of $h,h'$ and $\E_\pi(I+Q)<\infty$ control the state terms in \cref{eq:weighted-boundary-laplace-balance}.
For $h(q)=\e^{\theta q}$, choose $\theta<\bar\theta<\theta_c$; the bound $\E_\pi[(1+I)\e^{\bar\theta Q}]<\infty$ from \Cref{lem:stationary-moment-bounds} also gives $\E_\pi[Q\e^{\theta Q}]<\infty$.
In both cases, \cref{eq:regulator-occupation} makes $\eta(h)$ and $\eta(|h'|)$ finite.
These bounds are uniform for $s\ge1$; since $I>0$ almost surely, dominated convergence gives
\begin{equation}\label{eq:weighted-boundary-laplace-limit}
\lim_{s\to\infty}s^{\alpha-1}\E_\pi[\e^{-sI}h(Q)]
=\frac{\E_\pi[Qh(Q)]}{\kappa_\alpha}.
\end{equation}

To recover the boundary probabilities, define finite measures
\[
\mu_{\varepsilon,h}(\dd u)
:=\varepsilon^{1-\alpha}\e^{-u}\E_\pi[h(Q)\1_{\{I/\varepsilon\in\dd u\}}].
\]
By \cref{eq:weighted-boundary-laplace-limit}, their Laplace transforms at $z\ge0$ tend to $[\E_\pi Qh(Q)/\kappa_\alpha](1+z)^{1-\alpha}$, including convergence of total mass at $z=0$.
For positive limiting mass, normalize by total mass and apply the continuity theorem for Laplace transforms (\cite{Feller1971}, Chapter~XIII, Section~1).
Multiplying by the convergent total masses gives weak convergence to the finite measure $\frac{\E_\pi[Qh(Q)]}{\kappa_\alpha\Gamma(\alpha-1)}u^{\alpha-2}\e^{-u}\dd u$.
If the limiting mass is zero, convergence follows directly from the total masses.
For each $x>0$, integrating the bounded function $\e^u\1_{[0,x]}(u)$ gives
\begin{equation}\label{eq:weighted-boundary-probability}
\lim_{\varepsilon\downarrow0}\varepsilon^{1-\alpha}
\E_\pi[h(Q);I\le x\varepsilon]
=\frac{x^{\alpha-1}\E_\pi[Qh(Q)]}{\kappa_\alpha\Gamma(\alpha)}.
\end{equation}
The interval endpoints have zero limiting mass, justifying the integration.
For $h=1$ and $x=1$, the result is \cref{eq:small-idleness-asymptotic}.
For compactly supported $h$, division by this limit shows that the conditional queue laws converge vaguely to $q\pi_Q(\dd q)/m_Q$.
Since this limiting measure has total mass one, the convergence is weak.
Under the conditioning, $I/\varepsilon$ lies in $[0,1]$.
Taking $0<x\le1$ in \cref{eq:weighted-boundary-probability} identifies every joint subsequential limit through $x^{\alpha-1}\E h(Q^{\mathrm{sb}})$.
The limit factors as $\Pp(B\le x)\E h(Q^{\mathrm{sb}})$, identifying the product law in \cref{eq:boundary-product-limit}. \hfill
\end{proof}

The limiting queue law is the normalized regulator occupation measure, here obtained by conditioning stationary observations on $I\le\varepsilon$ and letting $\varepsilon\downarrow0$.

The boundary asymptotic controls negative moments of $I$, while the upper tail controls positive moments.
\begin{proof}[Proof of \Cref{cor:real-idleness-moments}]
For $a>0$, Tonelli's theorem gives
\[
\E_\pi[I^{-a};I\le1]
=\Pp_\pi(I\le1)+a\int_0^1x^{-a-1}\Pp_\pi(I\le x)\dd x.
\]
By \cref{eq:small-idleness-asymptotic}, the integrand near zero is bounded above and below by positive constant multiples of $x^{\alpha-a-2}$.
The integral is finite exactly when $a<\alpha-1$; the contribution from $I>1$ is bounded by one.
For $r>0$, the identity
\[
\E_\pi[I^r;I>1]
=\Pp_\pi(I>1)+r\int_1^\infty x^{r-1}\Pp_\pi(I>x)\dd x
\]
and \cref{eq:idleness-tail} give finiteness exactly when $r<\alpha$.
The complementary regions contribute bounded terms, and $\E_\pi I^0=1$.
This proves \cref{eq:real-idleness-moments}, including divergence at both endpoints.

For the endpoint limits, write $c_0:=m_Q/[\kappa_\alpha\Gamma(\alpha)]$ and $c_\infty:=\cA/\alpha^2$.
The boundary and upper-tail equivalents give $\Pp_\pi(I\le x)\sim c_0x^{\alpha-1}$ as $x\downarrow0$ and $\Pp_\pi(I>x)\sim c_\infty x^{-\alpha}$ as $x\to\infty$.
Fix $\varepsilon\in(0,1)$ and choose $0<x_0<1<x_1$ so that the ratios $\Pp_\pi(I\le x)/(c_0x^{\alpha-1})$ and $\Pp_\pi(I>x)/(c_\infty x^{-\alpha})$ lie in $[1-\varepsilon,1+\varepsilon]$ for $0<x\le x_0$ and $x\ge x_1$, respectively.
Apply the two truncated tail-integral formulas with cutoffs $x_0$ and $x_1$.
The contributions from $I>x_0$ in the negative-moment calculation and from $I\le x_1$ in the positive-moment calculation remain uniformly bounded as $a\uparrow\alpha-1$ and $r\uparrow\alpha$, respectively.
The boundary terms in the truncated formulas are also uniformly bounded, so these terms vanish after multiplication by the respective prefactors.
Keeping the cutoffs fixed, we have
\[
(\alpha-1-a)\int_0^{x_0}x^{\alpha-a-2}\dd x
=x_0^{\alpha-1-a}\to1,
\qquad
(\alpha-r)\int_{x_1}^\infty x^{r-\alpha-1}\dd x
=x_1^{r-\alpha}\to1.
\]
The tail comparisons bound the lower and upper limits of the normalized negative moments between $(1-\varepsilon)(\alpha-1)c_0$ and $(1+\varepsilon)(\alpha-1)c_0$, and those of the normalized positive moments between $(1-\varepsilon)\alpha c_\infty$ and $(1+\varepsilon)\alpha c_\infty$.
Letting $\varepsilon\downarrow0$ gives \cref{eq:inverse-idleness-pole,eq:positive-idleness-pole}. \hfill
\end{proof}

\section{Conclusion}\label{sec:discussion}

We have shown that heavy-tailed idleness, a light-tailed queue, and exponential convergence to equilibrium coexist in this degenerate reflected system.
The different stationary tails reflect the way input fluctuations enter the two coordinates: positive jumps create idleness, while queue growth occurs through the regulator at zero idleness.
The queue-tail equivalent further separates the exponential rate determined by the input and spare capacity from the polynomial correction imposed by the oblique reflection direction.

Reflection also links the convergence argument to the stationary dependence structure.
Boundary contact transmits randomness to the queue coordinate and makes two-dimensional minorization possible.
In stationarity, normalized regulator occupation and small-idleness conditioning identify the same size-biased queue law.
The large-queue conditional limit reveals another aspect of this dependence: typical idleness is of smaller order than its conditional mean, so the mean alone does not describe the scale of residual idle capacity during extreme congestion.

The stable-path construction provides a complementary view of how stationary performance changes with spare capacity.
With the input parameters fixed, coupling stationary states through a common stable path shows that increasing $\beta$ reduces the queue in stochastic order and yields diminishing reductions in its mean.
These comparisons describe capacity effects across the stationary queue distribution and give a direct interpretation of the construction.

The pathwise and regenerative representations provide starting points for numerical evaluation with controlled errors and for studying the stationary queue mean and tail constant as spare capacity becomes small or large.
For the queueing application, a central next step is to establish the reflected SDE as a limit of JSQ systems with heavy-tailed renewal arrivals and to prove convergence of their scaled stationary distributions.
Quantitative approximation bounds would then connect the equilibrium laws and capacity comparisons obtained here to the performance of finite-server systems.

\appendix
\crefalias{section}{appendix}
\crefalias{subsection}{appendix}
\crefalias{subsubsection}{appendix}

\section{Pathwise comparison and fixed-time minorization}\label{app:fixed-time-minorization}

We first prove \Cref{prop:initial-nonexpansiveness}, then use this comparison to control perturbations of the selected jump times and quantify how increasing the first jump size decreases the terminal queue.
These estimates provide the pathwise bounds needed for the two-jump construction of the fixed-time minorization in \Cref{prop:fixed-time-minorization}.
Recall $K_0=[0,\bar i_0]\times[0,\bar q_0]$ with $\bar i_0>\beta$ and $\bar q_0>0$.
Two selected jump sizes control the terminal queue and idleness through a triangular map.
We establish its geometry, and then use the L\'evy--It\^o decomposition to obtain a local minorization near $(\beta,0)$.
Deterministic attraction extends the bound to $K_0$.

\subsection{Nonexpansiveness in the initial state}\label{app:initial-nonexpansiveness}

\begin{proof}[Proof of \Cref{prop:initial-nonexpansiveness}]
The argument is pathwise for any common c\`adl\`ag input $w$ with no negative jumps.
Each regulator is continuous: if $\Delta U_t>0$, then $I_t=I_{t-}+\Delta w_t+\Delta U_t>0$, contradicting the regulator condition \cref{eq:regulator-condition}.
Set $x=I^y-I^{\widetilde y}$, $z=Q^y-Q^{\widetilde y}$, and $D=U^y-U^{\widetilde y}$.
The differences are continuous and of finite variation, with $\dd x=-(x+z)\dd t+\dd D$ and $\dd z=-z\dd t+\dd D$.
Inequalities between regulator differentials mean inequalities of Lebesgue--Stieltjes measures, restricted to the indicated sets where appropriate.
The chain rule for continuous finite-variation functions, with $\operatorname{sgn}(0)=0$, gives
\[
\dd(|x|+|z|)
=-\bigl[|x|+\operatorname{sgn}(x)z+|z|\bigr]\dd t
+[\operatorname{sgn}(x)+\operatorname{sgn}(z)]\dd D.
\]
On $\{x>0\}$, the first idleness is positive, so $\dd D\le0$; its coefficient is nonnegative.
On $\{x<0\}$, the second idleness is positive, so $\dd D\ge0$; its coefficient is nonpositive.
The coefficient of $\dd t$ on either set is at most $-|x|$.
For a continuous finite-variation function, its differential restricted to any level set is zero.
On $\{x=0\}$ this gives $\dd D=z\dd t$, hence $\dd z=0$ and no change in $|x|+|z|$.
Integrating proves \cref{eq:initial-nonexpansive}.
The estimate holds for all initial states on each such input path, hence simultaneously almost surely for the stable input. \hfill
\end{proof}

\subsection{Reference dynamics and queue sensitivity}

For the minorization construction, write $\Stable(t)=M^{(\varepsilon)}(t)-m_\varepsilon t+J^{(\varepsilon)}(t)$ as in \cref{eq:stable-levy-decomposition}, where $M^{(\varepsilon)}$ contains the compensated jumps of size at most $\varepsilon$, $J^{(\varepsilon)}$ contains the larger jumps, and $m_\varepsilon=\cA\varepsilon^{1-\alpha}/(\alpha-1)$.
Put $d:=m_\varepsilon-\beta$.
As $\varepsilon\downarrow0$, $d\to\infty$ while $\int_0^\varepsilon z^2\nu_{\mathrm A}(\dd z)\to0$.
With the small-jump path set to zero, the reference dynamics between selected jumps are
\[
(\dot I,\dot Q)=(-d-I-Q,-Q)\quad\text{when }I>0,
\qquad
(\dot I,\dot Q)=(0,d)\quad\text{when }I=0.
\]
From a post-jump state $(i,q)$, reflection-free evolution over $u$ time units gives
\[
Q(u)=q\e^{-u},\qquad I(u)=\e^{-u}(i+d)-d-qu\e^{-u}.
\]

Initial-state perturbations and jump-time shifts satisfy the following estimates.
By \cref{eq:fixed-time-solution-lipschitz}, perturbing the initial state and adding a path $m$ changes the solution by at most $C_T(\|y-\bar y\|+\|m\|_T)$ on $[0,T]$, with $C_T$ independent of the additive drift.
To vary jump times, suppose the reference initial state $\bar y$ and jump sizes $r_j$ satisfy $\|\bar y\|+\sum_j r_j=O(1+d)$.
Its speed between jumps is then $O(1+d)$: $I_t\le I_0+\sum_j r_j$ and $Q_t\le Q_0+dt$.
Moving a jump from $s_j^*$ to $s_j$ changes the state after both copies of that jump by at most $C(1+d)|s_j-s_j^*|$.
Nonexpansiveness in \Cref{prop:initial-nonexpansiveness} propagates this error through subsequent common input.
Thus, after the corresponding jumps have occurred in both paths, the state error is bounded by
\begin{equation}\label{eq:fixed-time-time-perturbation}
C_T(\|y-\bar y\|+\|m\|_T)+C(1+d)\sum_j|s_j-s_j^*|.
\end{equation}
Comparing at equal elapsed times after a jump adds only reference evolution over the time displacement, so the same bound holds after increasing $C$.
These constants are uniform for bounded horizons and fixed bounds on the reference initial state and jump sizes divided by $1+d$.

Changing the first jump size changes only idleness immediately after that jump; the next result quantifies the subsequent queue separation.
\begin{lemma}[Queue sensitivity to an idleness displacement]\label[lemma]{lem:excursion-queue-sensitivity}
Consider two solutions with a common c\`adl\`ag input having no negative jumps, $Q_s^{(1)}=Q_s^{(2)}$, and $I_s^{(2)}-I_s^{(1)}=\delta>0$.
Then $0\le Q_t^{(1)}-Q_t^{(2)}\le\delta$ for $t\ge s$.
If the second trajectory first reaches $I=0$ at $\tau\le s+t_{\mathrm r}$ with $t_{\mathrm r}<1$, then
\begin{equation}\label{eq:fixed-time-common-input-separation}
Q_t^{(1)}-Q_t^{(2)}\ge \e^{-(t-\tau)}(1-t_{\mathrm r})\delta,
\qquad t\ge\tau.
\end{equation}
\end{lemma}
\begin{proof}
Set $x=I^{(2)}-I^{(1)}$, $z=Q^{(1)}-Q^{(2)}$, and $D=U^{(1)}-U^{(2)}$.
The difference equations and \Cref{prop:initial-nonexpansiveness} give
\[
x_t+z_t=\delta-\int_s^t x_u\dd u,
\qquad
|x_t|+|z_t|\le\delta-\int_s^t|x_u|\dd u.
\]
Subtracting yields $2(x_t^-+z_t^-)\le-2\int_s^t x_u^-\dd u\le0$, so $x,z\ge0$ and $x+z\le\delta$.
On $\{x>0\}$, $I^{(2)}>0$ and $\dd U^{(2)}=0$, hence $\dd D\ge0$.
On $\{x=0\}$, the continuous finite-variation identity $\dd x=(-x+z)\dd t-\dd D$ gives $\dd D=z\dd t\ge0$, since the differential of $x$ vanishes on its zero set.
At $\tau$, both idleness coordinates vanish and $z_\tau=\delta-\int_s^\tau x_u\dd u\ge(1-t_{\mathrm r})\delta$.
Finally, $\dd z=-z\dd t+\dd D$ gives $z_t\ge\e^{-(t-\tau)}z_\tau$. \hfill
\end{proof}

\subsection{Two-jump maps and local minorization}

Write $y_*=(\beta,0)$.
For an initial state $y$, a horizon $T$, and times $0<s_1<s_2<T$, let $\Psi(a,b)$ be the terminal state under the input
\[
z(t)=-(\beta+d)t+m(t)+a\1_{[s_1,T]}(t)+b\1_{[s_2,T]}(t).
\]
The map's dependence on $y,m,s_1,s_2,T,d$ is suppressed.
We choose the jump times and size intervals so that the first jump controls the terminal queue, while the second controls terminal idleness during a reflection-free interval.

\begin{lemma}[Uniform two-jump terminal map]\label[lemma]{lem:fixed-time-two-jump-minorization}
There exist $r,\delta,\kappa>0$ such that, for every sufficiently large $d$, one can choose $0<T<3$, compact intervals $A_d,B_d\Subset(1,\infty)$ of positive length, nonempty open intervals $W_1,W_2\Subset(0,T)$ with $\sup W_1<\inf W_2$, and a nonempty open rectangle $\mathcal G_d\Subset(0,\infty)^2$ with the following property.
For every $y\in E$ with $\|y-y_*\|\le r$, every c\`adl\`ag $m$ with $m(0)=0$, no negative jumps, and $\|m\|_T\le\delta$, and every $s_j\in W_j$, the terminal map on $A_d\times B_d$ satisfies
\begin{equation}\label{eq:fixed-time-terminal-state-form}
\Psi(a,b)=(\iota(a)+cb,q(a)),\qquad
\kappa\le c\le1,\qquad
\mathcal G_d\subset\Psi(A_d\times B_d),
\end{equation}
where
\begin{equation}\label{eq:fixed-time-terminal-queue-sensitivity}
\kappa(a'-a)\le q(a)-q(a')\le a'-a,\qquad a<a',\quad a,a'\in A_d.
\end{equation}
\end{lemma}
\begin{proof}
\textit{Reference map.} Fix $0<\tau_0<\tau_{\mathrm r}<\tau_1<1$ with $\e^{\tau_0}-1<\tau_{\mathrm r}$, and $s_{\mathrm f}\in(0,1)$.
From $y_*$, the reference path reaches the face at $\log(1+\beta/d)$ and then has $\dot Q=d$.
It reaches $(0,1)$ at
\[
s_1^*:=\log(1+\beta/d)+1/d.
\]
Take $s_2^*=s_1^*+\tau_1$ and $T=s_2^*+s_{\mathrm f}<3$ for large $d$.
A first jump $a$ at $s_1^*$ returns to the face after the unique time $\tau(a)$ satisfying
\[
a+d=d\e^{\tau(a)}+\tau(a).
\]
Choose $a_d=d(\e^{\tau_0}-1)+\tau_0$ and $A_d=[a_d-1,a_d+1]$.
Uniformly on $A_d$, $\tau(a)\le a/d<\tau_{\mathrm r}$ for large $d$.
The reference terminal queue after a reflection-free final interval is
\[
\bar q(a):=\e^{-s_{\mathrm f}}\bigl[\e^{-\tau(a)}+d(\tau_1-\tau(a))\bigr],
\qquad
\bar q'(a)=-\e^{-(s_{\mathrm f}+\tau(a))}.
\]
Put $q_d:=\bar q(a_d)$ and choose $r_q>0$ with $4r_q<\e^{-(s_{\mathrm f}+\tau_{\mathrm r})}$.
Then $\bar q(A_d)$ contains $[q_d-4r_q,q_d+4r_q]$.

Choose $r_b>0$ and $i_*>0$ so that
\[
\e^{-s_{\mathrm f}}r_b>s_{\mathrm f}(r_q+1)+2,
\qquad
m_*:=i_*-\e^{-s_{\mathrm f}}r_b-s_{\mathrm f}\e^{-s_{\mathrm f}}>0.
\]
Set $b_d=\e^{s_{\mathrm f}}(i_*+d+s_{\mathrm f}q_d)-d$ and $B_d=[b_d-r_b,b_d+r_b]$.
A second jump $b$ at $s_2^*$ gives reference terminal idleness
\[
\bar i(a,b):=i_*+\e^{-s_{\mathrm f}}(b-b_d)-s_{\mathrm f}(\bar q(a)-q_d).
\]
Since $|\bar q(a)-q_d|\le\e^{-s_{\mathrm f}}$, this is at least $m_*$ on $A_d\times B_d$.
The unreflected idleness multiplied by $\e^u$ is strictly decreasing, so terminal positivity makes the entire final interval reflection-free.
Idleness then decreases and stays at least $m_*$.
Both jump-size intervals have lower endpoints tending to infinity and upper endpoints $O(1+d)$.

\textit{Perturbation bounds.} Choose $0<\eta<\min\{r_q,1,m_*/2\}$.
First choose $r,\delta>0$ small enough, independently of $d$, then choose windows about $s_1^*,s_2^*$ of sufficiently small width proportional to $(1+d)^{-1}$.
Keep $W_1$ after the reference first boundary hit and $s_2-s_1>\tau_{\mathrm r}$.
Before the first selected jump, compare with the no-jump reference path, which lies on the face throughout $W_1$.
For the terminal states, use \cref{eq:fixed-time-time-perturbation}.
These comparisons give, uniformly over the allowed states, paths, times, and jump sizes,
\begin{equation}\label{eq:two-jump-uniform-error}
I_{s_1-}\le\eta,\qquad
|I_T(a,b)-\bar i(a,b)|\le\eta,\qquad
|Q_T(a,b)-\bar q(a)|\le\eta.
\end{equation}
Applying \cref{eq:fixed-time-time-perturbation} at equal elapsed times after the second jump, together with the $O(1+d)$ bound on the reference speed, gives
\[
\inf_{s_2\le t\le T}I_t
\ge m_*-C(r+\delta)-C(1+d)\sum_{j=1}^2|s_j-s_j^*|
\ge m_*/2.
\]
The timing term also covers the change in final-interval length from $s_{\mathrm f}$ to $T-s_2$.
All constants are independent of $d$, so the tolerances and windows can be chosen to satisfy these bounds simultaneously.
The regulator is therefore constant on $[s_2,T]$.
Thus $Q_T=q(a)$ is independent of $b$, and $I_T=\iota(a)+cb$ with $c=\e^{-(T-s_2)}\in[\e^{-3},1]$.

Every first-jump excursion returns before $s_1+\tau_{\mathrm r}$.
Indeed, up to that return,
\[
I_{s_1+u}\le\eta+a_d+1-du+2\delta,
\]
and the right side is negative at $u=\tau_{\mathrm r}$ for large $d$.
This bound includes any perturbation jump at $s_1$.
Compare first-jump sizes $a<a'$ with the same second jump.
By \Cref{lem:excursion-queue-sensitivity}, $q$ is nonincreasing and one-Lipschitz.
Equation \cref{eq:fixed-time-common-input-separation} gives the lower bound in \cref{eq:fixed-time-terminal-queue-sensitivity} with $\kappa=\e^{-3}(1-\tau_{\mathrm r})$, which also bounds $c$ from below.

\textit{Common image rectangle.} Set
\[
\mathcal G_d:=(i_*-1,i_*+1)\times(q_d-r_q,q_d+r_q).
\]
For large $d$, this rectangle has compact closure in $(0,\infty)^2$.
By \cref{eq:two-jump-uniform-error}, $q(A_d)$ contains its queue interval.
Given a queue value $v$ in that interval, choose the unique $a\in A_d$ with $q(a)=v$.
Then $|\bar q(a)-q_d|<r_q+\eta<r_q+1$.
At $b=b_d-r_b$, the actual terminal idleness is below $i_*-1$; at $b=b_d+r_b$, it is above $i_*+1$, by the choice of $r_b$ and the error bound $\eta<1$.
Its affine dependence on $b$ therefore gives \cref{eq:fixed-time-terminal-state-form}. \hfill
\end{proof}

The triangular map yields the Lebesgue lower bound needed for the local minorization.
Substitution first in $b$, then monotonically in $a$ through $v=q(a)$, gives
\begin{equation}\label{eq:fixed-time-two-jump-integration}
\int_{A_d}\int_{B_d}\1_{\{\Psi(a,b)\in H\}}\dd b\dd a
\ge\operatorname{Leb}_2(H\cap\mathcal G_d),\qquad H\in\mathcal B(E).
\end{equation}
Indeed, the affine coefficient $c$ is at most one, and $q$ is decreasing and one-Lipschitz.
The inverse substitutions therefore do not decrease length; the image inclusion supplies every slice of $H\cap\mathcal G_d$.
Bi-Lipschitz continuity justifies the monotone substitution.

\begin{samepage}
For $r>0$, let $B_*(r):=\{y\in E:\|y-(\beta,0)\|\le r\}$.

\begin{proposition}[Local fixed-time minorization]\label[proposition]{prop:local-fixed-time-minorization}
There exist $r,\varepsilon_*,T_*>0$ and a nonempty open rectangle $\mathcal G\Subset(0,\infty)^2$ such that
\[
P_{T_*}(y,H)\ge\varepsilon_*\nu_{\mathcal G}(H),
\qquad y\in B_*(r),\quad H\in\mathcal B(E),
\]
where $\nu_{\mathcal G}$ is the uniform probability measure on $\mathcal G$.
\end{proposition}
\end{samepage}
\begin{proof}
Take $r,\delta$ from \Cref{lem:fixed-time-two-jump-minorization}.
For sufficiently small $\varepsilon\in(0,1)$, $d=m_\varepsilon-\beta$ is in the lemma's range, $A_d,B_d\subset(\varepsilon,\infty)$, and the small-jump event $\mathcal E:=\{\|M^{(\varepsilon)}\|_3\le\delta\}$ satisfies
\[
\Pp(\mathcal E)\ge1-\frac{12}{\delta^2}\int_0^\varepsilon z^2\nu_{\mathrm A}(\dd z)>0
\]
by Doob's inequality and the vanishing variance integral.
Require exactly one large jump in each of $W_1\times A_d$ and $W_2\times B_d$, and no other jump larger than $\varepsilon$ on $[0,T]$.
Write $\lambda_\varepsilon:=\nu_{\mathrm A}((\varepsilon,\infty))$ and let $n_A,n_B>0$ be the minima of $\cA z^{-1-\alpha}$ on $A_d,B_d$.
Almost surely, the small-jump path has no negative jumps, so on $\mathcal E$ it meets the terminal-map lemma's hypotheses.
Independence of the L\'evy components, measurability of the solution map, the Poisson density, and \cref{eq:fixed-time-two-jump-integration} give
\[
P_T(y,H)\ge
\Pp(\mathcal E)\e^{-\lambda_\varepsilon T}|W_1||W_2|n_An_B
\operatorname{Leb}_2(H\cap\mathcal G_d),
\qquad y\in B_*(r).
\]
All factors are positive.
Take $T_*=T$, $\mathcal G=\mathcal G_d$, and normalize Lebesgue measure to obtain $\varepsilon_*$. \hfill
\end{proof}

\subsection{Extension to the reference compact set}

\begin{proof}[Proof of \Cref{prop:fixed-time-minorization}]
Use $r$ and $T_*$ from \Cref{prop:local-fixed-time-minorization}.
By \Cref{lem:deterministic-attraction}, choose $T_0$ such that $\sup_{y\in K_0}\|\varphi_{T_0}(y)-y_*\|<r/2$.
We apply \cref{eq:fixed-time-solution-lipschitz} on the small-ball event in \cref{eq:stable-small-ball} to obtain
\[
\inf_{y\in K_0}P_{T_0}(y,B_*(r))
\ge\Pp(\|\Stable\|_{T_0}<r/(2C_{T_0}))=:p_0>0.
\]
The semigroup property and the local minorization yield \cref{eq:fixed-time-minorization-main} with $\Delta=T_0+T_*$ and $\varepsilon_0=p_0\varepsilon_*$. \hfill
\end{proof}

\section{Estimates for the sharp queue tail}\label{app:queue-tail-estimates}

This appendix supplies the estimates used in the regenerative queue-tail analysis of \Cref{sec:queue-tail-analysis}.
We use the downcrossing bounds in \Cref{app:queue-downcrossing} to prove finite mean cycle lengths in \Cref{lem:queue-regeneration} and justify the occupation limit in \Cref{lem:high-level-occupation}.
These bounds also ensure that the exit time in the proof of \Cref{lem:sharp-queue-hitting} is almost surely finite.
In \Cref{app:critical-barriers}, we complete the construction of the barriers in \Cref{lem:critical-barriers}.

\subsection{Bounds before a queue downcrossing}\label{app:queue-downcrossing}

\begin{lemma}[Exponential bounds before a queue downcrossing]\label[lemma]{lem:killed-decay}
For each $\theta_0\in(0,\theta_c)$ there are constants $q_{\mathrm L},C,\delta>0$ such that, for $\ell>q_{\mathrm L}$, $i\ge0$, $q>\ell$, and $t\ge0$,
\begin{subequations}
\begin{align}
\Pp_{(i,q)}(\tau_\ell^\downarrow>t)
&\le C(1+i)\e^{\theta_0(q-\ell)}\e^{-\delta t},\label{eq:killed-exit-bound}\\
\sup_{x>\ell}\Pp_{(0,x)}(Q_t>x,\ t<\tau_\ell^\downarrow)
&\le C\e^{-\delta t}.\label{eq:killed-occupation-bound}
\end{align}
\end{subequations}
Since $Q-I\le Q$, \cref{eq:killed-occupation-bound} also holds with $Q_t-I_t>x$ in place of $Q_t>x$.
\end{lemma}
\begin{proof}
Use $V=V_{1,\theta_0}=\e^{\theta_0q}g(i)$, $\delta>0$, and $K:=K_{1,\theta_0}=[0,i_{\mathrm L}]\times[0,q_{\mathrm L}]$ from \Cref{prop:foster-lyapunov}.
The region $q>\ell$ is disjoint from $K$, so $(\gen+\delta)V\le0$ there.
Applying the stopped It\^o formula \cref{eq:stopped-dynkin-formula} and the product rule shows that $\e^{\delta(t\wedge\tau_\ell^\downarrow)}V(Y_{t\wedge\tau_\ell^\downarrow})$ is a nonnegative local supermartingale, hence a supermartingale.
Discarding the nonnegative stopped term on $\{\tau_\ell^\downarrow\le t\}$ gives
\begin{equation}\label{eq:killed-lyapunov-bound}
\E_{(i,q)}[V(Y_t);\ t<\tau_\ell^\downarrow]
\le\e^{-\delta t}V(i,q).
\end{equation}
By \cref{eq:lyapunov-weight-comparison}, $g_*:=\inf_{i\ge0}g(i)>0$ and $g(i)\le C_0(1+i)$.
Applying \cref{eq:killed-lyapunov-bound} with $V(Y_t)\ge g_*\e^{\theta_0\ell}$ before the downcrossing proves \cref{eq:killed-exit-bound}.
For \cref{eq:killed-occupation-bound}, use $V(Y_t)\ge g_*\e^{\theta_0x}$ on $\{Q_t>x\}$ and $V(0,x)=\e^{\theta_0x}$, choosing $C\ge g_*^{-1}\max\{C_0,1\}$. \hfill
\end{proof}

\subsection{Uniform estimates for the critical barriers}\label{app:critical-barriers}

We complete the construction in the proof of \Cref{lem:critical-barriers} by choosing the constants in \cref{eq:critical-barrier-construction} and proving the stated bounds.
Throughout, $\gamma=\alpha/2$.
For polynomials $f$, define
\[
\jumpgen_cf(i):=\int_0^\infty[f(i+z)-f(i)-zf'(i)]\e^{-\theta_cz}\nu_{\mathrm A}(\dd z).
\]
The operator $\jumpgen_c$ maps polynomials to polynomials because the exponentially weighted L\'evy measure $\e^{-\theta_cz}\nu_{\mathrm A}(\dd z)$ has finite moments of orders at least two.
In particular, $\jumpgen_c(i^2)=\sigma_c^2$, where $\sigma_c^2=\int_0^\infty z^2\e^{-\theta_cz}\nu_{\mathrm A}(\dd z)$.
For functions $u\in C^2(\R_+\times(0,\infty))$ that are polynomial in $i$, $\jumpgen_c$ acts on $i$ with $q$ fixed.
Since $\Phi(\theta_c)=0$ and $\int_0^\infty z(1-\e^{-\theta_cz})\nu_{\mathrm A}(\dd z)=\alpha\beta$, conjugation by $h_c$ gives
\begin{equation}\label{eq:critical-conjugacy}
h_c^{-1}\gen(h_cu)=\jumpgen_cu-(\mu_c+i+q)u_i-qu_q+\theta_ci\,u.
\end{equation}
For polynomials $f$ and integers $j\ge0$, set
\[
\mathcal C_jf:=\jumpgen_cf-(\mu_c+i)f'+(\gamma+j+\theta_ci)f.
\]
Each $\mathcal C_j$ acts on $i$, treating $\zeta$ as a fixed parameter when applied to the polynomials below.
Conjugacy in \cref{eq:critical-conjugacy} yields the exact identity
\begin{equation}\label{eq:critical-monomial-generator}
\gen\{h_c(i,q)q^{-\gamma-j}f(i)\}
=h_c(i,q)q^{-\gamma}
\{q^{-j}\mathcal C_jf(i)-q^{1-j}f'(i)\}.
\end{equation}
Since $\jumpgen_cc_1=\theta_c\sigma_c^2/2$, the relations in \cref{eq:critical-polynomial-recursion} give $c_1'=\gamma+\theta_ci$ and $c_2'=\mathcal C_1c_1$.
Substitution in \cref{eq:critical-monomial-generator} leaves
\[
\gen b_\zeta(i,q)=w_c(q)q^{-2}\e^{-\theta_ci}\mathcal C_2c_2(i;\zeta).
\]
The choice of $\gamma$ in \cref{eq:critical-tail-power} gives $(b_\zeta)_i(0,q)+(b_\zeta)_q(0,q)=0$.

The parameter $\zeta$ determines the sign of $\gen b_\zeta(0,q)$:
\[
\partial_\zeta c_2(i;\zeta)=(\gamma+1)i+\frac{\theta_c}{2}i^2,
\qquad
\partial_\zeta[\mathcal C_2c_2(0;\zeta)]
=-\mu_c(\gamma+1)+\frac{\theta_c\sigma_c^2}{2}=-\mu_c.
\]
Set $\zeta_\pm:=\frac{[\mathcal C_2c_2(\cdot;0)](0)\pm1}{\mu_c}$.
Thus $\mathcal C_2c_2(0;\zeta_\pm)=\mp1$.

To obtain these signs for every idleness value, consider the bounded factor $\psi_0$ in \cref{eq:critical-barrier-construction}.
The function $\psi_0'+\theta_c\psi_0$ is nonnegative, strictly positive for $i>0$, and tends to $\theta_c$ at infinity.
The functions $\psi_0$, $\psi_0'$, $i\psi_0'$ and $\psi_0''$ are bounded, and \Cref{lem:jump-operator-estimates} gives boundedness of $\jumpgenI\psi_0$.
Thus the following expression is uniformly bounded for $i\ge0$:
\[
\jumpgenI\psi_0(i)+(\beta-i)\psi_0'(i)+(\gamma+3)\psi_0(i).
\]
For each fixed $C_{\mathrm b}>0$, the generator of the barriers in \cref{eq:critical-barrier-construction} satisfies
\begin{equation*}
\frac{\gen f_\pm(i,q)}{w_c(q)}
=\frac{\e^{-\theta_ci}\mathcal C_2c_2(i;\zeta_\pm)
\mp C_{\mathrm b}[\psi_0'(i)+\theta_c\psi_0(i)]}{q^2}
+O(q^{-3}).
\end{equation*}
The remainder equals $\pm C_{\mathrm b}q^{-3}[\jumpgenI\psi_0(i)+(\beta-i)\psi_0'(i)+(\gamma+3)\psi_0(i)]$.
The expression in brackets is bounded on $\R_+$, so the remainder is uniform in $i\ge0$.
Choose $i_0>0$ so that the exponentially damped polynomial in the coefficient of $q^{-2}$ is at least $1/2$ for $f_-$ and at most $-1/2$ for $f_+$ on $[0,i_0]$.
The correction strengthens these signs for every $C_{\mathrm b}>0$.
On $[i_0,\infty)$, $\psi_0'+\theta_c\psi_0$ has a positive minimum, whereas both exponentially damped polynomials are bounded.
Choose a fixed $C_{\mathrm b}$ large enough to make the coefficient of $q^{-2}$ at least $1/2$ for $f_-$ and at most $-1/2$ for $f_+$ throughout the half-line.
The uniform $O(q^{-3})$ remainder gives $\gen f_-\ge0$ and $\gen f_+\le0$ for all $i\ge0$ and sufficiently large $q$.

For positivity, $c_1(\cdot;\zeta_+)$ and $c_2(\cdot;\zeta_+)$ have degrees two and four, with positive leading coefficients $\theta_c/2$ and $\theta_c^2/8$.
Each polynomial is therefore bounded below on $[0,\infty)$, so $1+c_1(i;\zeta_+)/q+c_2(i;\zeta_+)/q^2\ge1/2$ for all sufficiently large $q$, uniformly in $i\ge0$.
Since $\psi_0\ge0$, this proves $f_+>0$.
Enlarge $q_{\mathrm b}$ so that $1+\zeta_\pm/q>0$ for $q\ge q_{\mathrm b}$, as required by \cref{eq:critical-barrier-boundary}.

On any strip $\R_+\times[q_{\mathrm b},x]$, both barriers and their first two idleness derivatives are bounded.
Their jump remainders are bounded by $C_xz^2$ for $z\le1$ and $C_x(1+z)$ for $z>1$, uniformly over the strip.
These bounds are integrable against $\nu_{\mathrm A}$, verifying the jump-integrability condition in \Cref{lem:critical-barriers}.

\section{Well-posedness and pathwise estimates}\label{app:pathwise}

For $w=(w_1,w_2)\in\Dspace([0,T],\R^2)$ with $w_1(0)\ge0$, define the oblique reflection maps
\[
\RegMap(w)(t):=\sup_{0\le s\le t}[-w_1(s)]^+,
\qquad
\SkMap(w)(t):=w(t)+(1,1)\RegMap(w)(t).
\]
The first coordinate is the one-sided Skorokhod map (\cite{Whitt2002}, Section~13.5); the second gives the oblique shift.

\begin{lemma}[Lipschitz continuity of the reflection map]\label[lemma]{lem:sk-map}
For $w,\widetilde w\in\Dspace([0,T],\R^2)$ with nonnegative first initial coordinates,
\[
\norm{\SkMap(w)-\SkMap(\widetilde w)}_T
\le3\norm{w-\widetilde w}_T.
\]
The pair $(\SkMap(w),\RegMap(w))$ is the unique solution of the Skorokhod problem on the half-space $[0,\infty)\times\R$ with reflection direction $(1,1)$, and
\[
\Delta\RegMap(w)_t=[-\SkMap(w)_1(t-)-\Delta w_1(t)]^+.
\]
\end{lemma}

\begin{proof}
The positive-part map and the running supremum are one-Lipschitz, so the estimate follows by the triangle inequality.
The function $\RegMap(w)$ is the least nondecreasing correction starting at zero that keeps $w_1+\RegMap(w)$ nonnegative, and $\dd\RegMap(w)$ is supported on $\{w_1+\RegMap(w)=0\}$.
These properties imply uniqueness of the one-dimensional Skorokhod solution.
The jump identity follows by comparing the running supremum of $(-w_1)^+$ immediately before and after a jump. \hfill
\end{proof}

Let $b(i,q):=(\beta-i-q,-q)$.
For initial state $y\in E$ and a c\`adl\`ag driving path $z$ with $z(0)=0$, the reflected equation is equivalent to the fixed-point equation
\begin{equation}\label{eq:pathwise-fixed-point}
Y=\SkMap\left(y+(z,0)+\int_0^{\cdot}b(Y_s)\dd s\right).
\end{equation}

\begin{lemma}[Pathwise solvability and Lipschitz estimates]\label[lemma]{lem:solution-lipschitz}
For every $T<\infty$, $y\in E$, and $z\in\Dspace([0,T],\R)$ with $z(0)=0$, \cref{eq:pathwise-fixed-point} has a unique solution $Y^{y,z}$ on $[0,T]$.
There is a constant $C_T<\infty$, depending only on $T$ and the fixed model parameters, such that for any two such initial-state/input pairs,
\begin{equation}\label{eq:fixed-time-solution-lipschitz}
\|Y^{y,z}-Y^{\widetilde y,\widetilde z}\|_T
\le C_T\bigl(\|y-\widetilde y\|+\|z-\widetilde z\|_T\bigr).
\end{equation}
\end{lemma}

\begin{proof}
The drift $b$ is globally Lipschitz.
On an interval where the product of its length and the drift and reflection-map Lipschitz constants is less than one, \cref{eq:pathwise-fixed-point} is a contraction in the uniform norm.
Iteration gives a unique solution on every finite interval.
Applying the reflection-map bound to two solutions and then Gronwall's inequality gives \cref{eq:fixed-time-solution-lipschitz}.
The same estimate against a fixed zero-input solution gives a finite-horizon linear-growth bound.
The successive approximations are measurable in the initial state and input path, so their limit is also measurable. \hfill
\end{proof}

\begin{proof}[Proof of \Cref{prop:well-posedness}]
Apply \Cref{lem:solution-lipschitz} to the stable L\'evy input $\Stable$.
The pathwise construction is adapted and measurable, and its finite-horizon growth bound rules out explosion.
It therefore gives a pathwise unique nonexplosive strong solution.

The first unreflected coordinate has no negative jumps.
The jump identity in \Cref{lem:sk-map} therefore gives $\Delta U=0$.
The representation \cref{eq:queue-variation-of-constants} gives nonnegativity and continuity of $Q$.
Measurability, pathwise uniqueness, and the strong Markov property of the L\'evy input give the strong Markov property of the solution: after a stopping time, the same solution map acts on the state and independent post-stopping input increments.

Let $C_0(E)$ denote the continuous functions that vanish at infinity.
The estimate \cref{eq:fixed-time-solution-lipschitz}, with a common input path, implies that $P_tf$ is continuous for every bounded continuous $f$.
To prove $P_tC_0(E)\subset C_0(E)$, fix $T<\infty$ and a driving path.
If the initial queue coordinate $q$ tends to infinity, then $Q_T\ge\e^{-T}q$.
For $0\le q\le M$, consider the candidate solution without reflection,
\[
\bar Q_s=q\e^{-s},\qquad
\bar I_s=\e^{-s}i+\beta(1-\e^{-s})-qs\e^{-s}+Z_s,
\qquad
Z_s:=\int_0^s\e^{-(s-u)}\dd\Stable(u).
\]
Integration by parts gives
\[
Z_s=\Stable(s)-\int_0^s\e^{-(s-u)}\Stable(u)\dd u,
\qquad \sup_{s\le T}|Z_s|\le2\|\Stable\|_T.
\]
Hence $\inf_{s\le T}\bar I_s\ge\e^{-T}i-MT-2\|\Stable\|_T>0$ for all sufficiently large $i$, uniformly over $q\le M$.
This candidate is then the reflected solution with $U\equiv0$, and $\sup_{s\le T}|I_s-\e^{-s}i|\le\beta+MT+2\|\Stable\|_T$.
Thus $Y_T^y$ leaves every compact set pathwise as $\|y\|\to\infty$.
For a general initial sequence, separate the cases where $q$ is large or lies in a fixed bounded interval.
Since $|f(Y_T^y)|\le\|f\|_\infty$ for $f\in C_0(E)$, dominated convergence gives $P_Tf(y)\to0$.

For strong continuity, first take $f\in C_c(E)$.
On each compact $K$, comparison with the zero-input flow gives
\[
\sup_{y\in K}\sup_{u\le t}\|Y_u^y-y\|
\le C_1\|\Stable\|_t+C_Kt\longrightarrow0
\quad\text{almost surely as }t\downarrow0.
\]
Here $t\le1$, and the zero-input bound follows from the locally bounded drift and the reflection-map estimate.
Uniform continuity of $f$ and bounded convergence therefore give $\sup_{y\in K}|P_tf(y)-f(y)|\to0$.

To control the complement, take $\operatorname{supp}f\subset[0,R]^2$, choose $M>\mathrm e R$ and $L>\mathrm e(R+M+2)$, and set $K=[0,L]\times[0,M]$.
For $t\le1$, on $\{\|\Stable\|_t\le1\}$, an initial queue above $M$ gives $Q_t>R$.
If $q\le M$ and $i>L$, the preceding unreflected bound gives $\inf_{s\le t}\bar I_s\ge\e^{-1}i-M-2>R$, so no reflection occurs and $I_t>R$.
Hence
\[
\sup_{y\notin K}|P_tf(y)|
\le\|f\|_\infty\Pp(\|\Stable\|_t>1)\longrightarrow0.
\]
Since $f=0$ outside $K$, this proves $\|P_tf-f\|_\infty\to0$ for $f\in C_c(E)$.
Uniform approximation by $C_c(E)$ and contractivity extend the conclusion to $C_0(E)$.
The process is therefore Feller--Dynkin, and its c\`adl\`ag version on the locally compact separable space $E$ is Borel right; see Sharpe (\cite{Sharpe1988}, Chapter~II, Section~19). \hfill
\end{proof}

\section{Properties of the stable input}\label{app:stable-input}
\subsection{Compensation and support properties}\label{app:stable-normalization}

The mean-zero strictly stable input satisfies
\[
\E\e^{\mathrm i u\Stable(t)}
=\exp\left\{t\int_0^\infty
(\e^{\mathrm iuz}-1-\mathrm iuz)\nu_{\mathrm A}(\dd z)\right\},
\qquad u\in\R.
\]
The linear subtraction makes the integrand $O(z^2)$ as $z\downarrow0$, as required because $\int_0^1z\nu_{\mathrm A}(\dd z)=\infty$.
For large jumps, the first moment is finite: $\int_1^\infty z\nu_{\mathrm A}(\dd z)=\frac{c_{\mathrm A}}{\alpha-1}$.
Compensating these jumps incorporates their mean into the centering of $\Stable$.
Equivalently, the L\'evy--Khintchine representation with compensation restricted to $z\le1$ has drift $-c_{\mathrm A}/(\alpha-1)$ \cite{Applebaum2009}:
\[
\log\E\e^{\mathrm i u\Stable(t)}
=t\left[-\frac{\mathrm i u c_{\mathrm A}}{\alpha-1}
+\int_0^\infty
\bigl(\e^{\mathrm iuz}-1-\mathrm iuz\1_{(0,1]}(z)\bigr)\nu_{\mathrm A}(\dd z)\right].
\]

We next verify the Laplace transform in \cref{eq:stable-laplace-exponent}.
For $\theta\ge0$, the L\'evy--Khintchine formula gives
\[
\log\E\e^{-\theta\Stable(t)}
=t\cA\int_0^\infty
(\e^{-\theta z}-1+\theta z)z^{-1-\alpha}\dd z.
\]
Substituting $u=\theta z$ and integrating by parts twice gives
\[
\int_0^\infty(\e^{-u}-1+u)u^{-1-\alpha}\dd u
=\frac{\Gamma(2-\alpha)}{\alpha(\alpha-1)}
=\Gamma(-\alpha),
\]
which proves \cref{eq:stable-laplace-exponent}.

Let $N(\dd s,\dd z)$ be the Poisson random measure of the jumps of $\Stable$, with intensity $\dd s\,\nu_{\mathrm A}(\dd z)$, and let $\widetilde N$ denote its compensated version.
For $\varepsilon>0$, use the L\'evy--It\^o decomposition
\begin{equation}\label{eq:stable-levy-decomposition}
\begin{gathered}
\Stable(t)=M^{(\varepsilon)}(t)-m_\varepsilon t+J^{(\varepsilon)}(t),
\qquad
m_\varepsilon:=\int_\varepsilon^\infty z\nu_{\mathrm A}(\dd z)
=\frac{\cA}{\alpha-1}\varepsilon^{1-\alpha},\\
M^{(\varepsilon)}(t):=\int_{[0,t]\times(0,\varepsilon]}z\,\widetilde N(\dd s,\dd z),
\qquad
J^{(\varepsilon)}(t):=\int_{[0,t]\times(\varepsilon,\infty)}z\,N(\dd s,\dd z).
\end{gathered}
\end{equation}
The two processes are independent, $M^{(\varepsilon)}$ is a square-integrable martingale, and $J^{(\varepsilon)}$ is compound Poisson.
Moreover, $\E[M^{(\varepsilon)}(T)^2]=T\int_0^\varepsilon z^2\nu_{\mathrm A}(\dd z)=T\cA\varepsilon^{2-\alpha}/(2-\alpha)$.

\begin{lemma}[Small-ball probabilities and negative fluctuations]\label[lemma]{lem:stable-support-estimates}
For every $T,\delta>0$,
\begin{equation}\label{eq:stable-small-ball}
\Pp\left(\sup_{0\le s\le T}\abs{\Stable(s)}<\delta\right)>0.
\end{equation}
For every $T,L>0$,
\begin{equation}\label{eq:stable-negative-tail}
\Pp(\Stable(T)<-L)>0.
\end{equation}
\end{lemma}

\begin{proof}
For \cref{eq:stable-small-ball}, first choose $\varepsilon$ so small that Doob's inequality gives
\[
\Pp(\sup_{s\le T}|M^{(\varepsilon)}(s)|<\delta/2)>0.
\]
Because $m_\varepsilon T/\varepsilon\to\infty$ as $\varepsilon\downarrow0$, choose an integer $n$ such that $a:=m_\varepsilon T/n\in(2\varepsilon,3\varepsilon)$.
Partition $[0,T]$ into $n$ equal intervals.
Fix $\delta_{\mathrm j}>0$ so small that $a-\delta_{\mathrm j}/n>\varepsilon$.
Require exactly one jump of $J^{(\varepsilon)}$ on each interval, with size in $[a-\delta_{\mathrm j}/n,a+\delta_{\mathrm j}/n]$, and no other jump larger than $\varepsilon$.
This compound-Poisson event has positive probability.
On it, the cumulative error at the partition points is at most $\delta_{\mathrm j}$, while within each interval the linear drift and the single jump create an additional deviation of at most $3\varepsilon+\delta_{\mathrm j}/n$.
Hence $\sup_{s\le T}|J^{(\varepsilon)}(s)-m_\varepsilon s|\le3\varepsilon+2\delta_{\mathrm j}$.
Choosing $\varepsilon$ and $\delta_{\mathrm j}$ so that $3\varepsilon+2\delta_{\mathrm j}<\delta/2$ and intersecting the two independent events proves \cref{eq:stable-small-ball}.
For general small-deviation theory, see Aurzada and Dereich \cite{AurzadaDereich2009}.

For \cref{eq:stable-negative-tail}, suppose $\Stable(T)\ge-L$ almost surely.
Then \cref{eq:stable-laplace-exponent} gives $T\kappa_\alpha\theta^\alpha\le\theta L$ for every $\theta>0$, contradicting $\alpha>1$ as $\theta\to\infty$. \hfill
\end{proof}

\subsection{Densities, tails, and stochastic integrals}

\begin{lemma}[Stable densities, tails, and stochastic integrals]\label[lemma]{lem:stable-convolution-tail}
For every $t>0$, $\Stable(t)$ has a bounded continuous density, and
\begin{equation}\label{eq:stable-unit-right-tail}
\lim_{x\to\infty}x^\alpha\Pp(\Stable(1)>x)=\frac{\cA}{\alpha}.
\end{equation}
For deterministic measurable $h:[0,\infty)\to[0,\infty)$ with $\int_0^\infty h(s)^\alpha\dd s<\infty$, the stable stochastic integral obeys
\begin{equation}\label{eq:stable-integral-scaling}
\int_0^\infty h(s)\dd\Stable(s)\overset d=
\left(\int_0^\infty h(s)^\alpha\dd s\right)^{1/\alpha}\Stable(1).
\end{equation}
For a two-sided driver, the same identity holds over $\R$ for deterministic nonnegative $h\in L^\alpha(\R)$.
\end{lemma}
\begin{proof}
The modulus of the characteristic function of $\Stable(t)$ is $\exp\{t\kappa_\alpha\cos(\pi\alpha/2)|u|^\alpha\}$.
Since $\cos(\pi\alpha/2)<0$, it is integrable, and Fourier inversion gives the density assertion.
The law of $\Stable(1)$ is infinitely divisible with regularly varying L\'evy measure $\nu_{\mathrm A}$.
Applying the results of Hult and Lindskog (\cite{HultLindskog2006}, Proposition~3.1 and Corollary~3.1) to the set $(1,\infty)$ gives
\[
\Pp(\Stable(1)>x)\sim\nu_{\mathrm A}((x,\infty))
=\frac{\cA}{\alpha}x^{-\alpha},
\]
which proves \cref{eq:stable-unit-right-tail}.

For simple $h\ge0$, independence of the increments and homogeneity of the stable characteristic exponent give \cref{eq:stable-integral-scaling}.
Approximation in $L^\alpha$ extends the identity to the stated class and constructs integrals over infinite time intervals.
Since $\alpha>1$, these integrals also converge in $L^1$. \hfill
\end{proof}

\section{Localization and stationary It\^o identities}\label{app:stationary-localization}
\subsection{Generator and Dynkin formulas}

Recall the exit times $\tau_m:=\inf\{t\ge0:I_t+Q_t\ge m\}$, $m\in\mathbb N$, used in \cref{eq:stopped-dynkin-formula}.

\begin{proposition}[Stopped extended-generator criterion]\label[proposition]{prop:generator-criterion}
Let $f\in\mathcal C_{\mathrm{ref}}^2(E)$ and fix a deterministic initial state $y\in E$.
For every $m\in\mathbb N$, the process $M^{f,m}$ in \cref{eq:stopped-dynkin-formula} is a martingale.
\end{proposition}

\begin{proof}
Apply It\^o's formula for jump semimartingales up to $t\wedge\tau_m$.
The regulator contribution vanishes because $f\in\mathcal C_{\mathrm{ref}}^2(E)$.
Set $K_m:=\{(i,q)\in E:i+q\le m\}$.
At jump times $0<s\le\tau_m$, the pre-jump state $Y_{s-}$ lies in $K_m$, including at a jump causing exit.
The integrability condition defining $\mathcal C_J^2(E)$ makes the compensator of the jump remainder integrable, uniformly over these pre-jump states.
For jumps of size at most one, bounded derivatives on $K_{m+1}$ and $\int_0^1z^2\nu_{\mathrm A}(\dd z)<\infty$ make the compensated integral square-integrable.
For jumps larger than one, the same integrability condition and $\int_1^\infty z\nu_{\mathrm A}(\dd z)<\infty$ give an integrable bound on $|f(y+(z,0))-f(y)|$, uniform over $y\in K_m$.
Each stopped stochastic integral is therefore a martingale.
Nonexplosion gives $\tau_m\uparrow\infty$ almost surely, so these stopping times show that $f(Y_t)-f(Y_0)-\int_0^t\gen f(Y_s)\dd s$ is a local martingale.
Thus $f$ belongs to the extended-generator domain in the local-martingale sense, with extended generator $\gen f$; see Applebaum (\cite{Applebaum2009}, Section~4.4). \hfill
\end{proof}

For the Lyapunov functions, the stopping can be removed to obtain the Dynkin identity used in \Cref{sec:harris}.
Fix $p\in(0,\alpha)$ and $\theta\in(0,\theta_c)$, and let $V$ be any scalar multiple of $V_{p,\theta}$ with $V\ge1$.

\begin{lemma}[Dynkin formula for the Lyapunov functions]\label[lemma]{lem:lyapunov-dynkin}
For every such $V$ and every deterministic initial state $y\in E$,
\[
V(Y_t)-V(Y_0)-\int_0^t\gen V(Y_s)\dd s
\]
is a martingale, and $\E_y\int_0^t|\gen V(Y_s)|\dd s<\infty$ for every $t<\infty$.
\end{lemma}
\begin{proof}
Choose $r>1$ and exponents $\bar p,\bar\theta$ with $rp<\bar p<\alpha$ and $r\theta<\bar\theta<\theta_c$, and let $\bar V$ be a scalar multiple of $V_{\bar p,\bar\theta}$ with $\bar V\ge1$.
The weight comparison and the bounded idleness coefficients in the generator ratio from \Cref{prop:foster-lyapunov} give
\[
V^r\le C\bar V,
\qquad
|\gen V|\le C(1+q)V\le C'\bar V.
\]
With $\tau_m$ as in \Cref{prop:generator-criterion}, the stopped drift inequality for $\bar V$ yields
\[
\E_y\bar V(Y_{t\wedge\tau_m})\le\bar V(y)+\bar b t,
\qquad m\ge1,\quad t\ge0.
\]
Thus the family $\{V(Y_{t\wedge\tau_m}):m\ge1\}$ is uniformly integrable for each fixed $t$.
Fatou's lemma also gives $\E_y\bar V(Y_s)\le\bar V(y)+\bar b s$, so $\E_y\int_0^t|\gen V(Y_s)|\dd s<\infty$.
Nonexplosion, this uniform integrability, and domination of the drift integrals allow passage to the limit in the stopped martingales in $L^1$.
Their limit is therefore a martingale. \hfill
\end{proof}

\subsection{Proof of the stationary It\^o identity}\label{app:stationary-ito-proof}

Allowing a nonzero oblique derivative adds a regulator term to It\^o's formula.
The occupation identity \cref{eq:regulator-occupation} expresses its stationary expectation through the queue marginal.

\begin{proof}[Proof of \Cref{prop:stationary-ito}]
The proof of \Cref{prop:generator-criterion}, with the regulator term retained, gives the stopped It\^o formula.
Start with law $\pi$ and take expectations.
For each fixed horizon $T$, boundedness of $f$ gives $f(Y_{T\wedge\tau_m})\to f(Y_T)$ in $L^1$ as $m\to\infty$.
Stationarity gives $\E_\pi\int_0^T|\gen f(Y_t)|\dd t=T\pi(|\gen f|)<\infty$.
By \cref{eq:regulator-occupation}, the expected absolute regulator integral is at most $T\int_0^\infty|f_i(0,q)+f_q(0,q)|q\pi_Q(\dd q)<\infty$.
These integrable path functionals dominate the stopped integrals.
By stationarity, $\E_\pi f(Y_T)=\E_\pi f(Y_0)$, which gives \cref{eq:stationary-ito-boundary}.

For the second assertion, choose a smooth cutoff $\chi$ equal to one on $[0,1]$ and zero on $[2,\infty)$, and set $h_N(q)=h(q)\chi(q/N)$, $N\ge1$.
The functions $h_N$ and their first two derivatives satisfy \cref{eq:stationary-exponential-test-class} with a constant independent of $N$ and converge pointwise to their counterparts for $h$.
The bounded functions $f_N(i,q)=\e^{-si}h_N(q)$ satisfy the first assertion.
For fixed $s>0$, their $\dd t$ integrands are bounded in absolute value by $C_s(1+q)\e^{\theta q}$, using $i\e^{-si}\le1/(\mathrm e s)$.
Their boundary terms are bounded by $C_s q\e^{\theta q}$.
A jump in idleness produces the remainder $h_N(q)\e^{-si}(\e^{-sz}-1+sz)$, whose absolute value is bounded by $C_s\e^{\theta q}z^2$ on $0<z\le1$ and $C_s\e^{\theta q}(1+z)$ on $z>1$.
For each fixed $q$, the small- and large-jump bounds are integrable in $z$ against $\nu_{\mathrm A}(\dd z)$.
Choose $\theta<\bar\theta<\theta_c$.
Since $(1+q)\e^{\theta q}\le C_{\theta,\bar\theta}\e^{\bar\theta q}$ and $\E_\pi\e^{\bar\theta Q}<\infty$ by \Cref{lem:stationary-moment-bounds}, these estimates give integrable bounds, independent of $N$, for the endpoint terms and the $\dd t$ and $\dd U$ integrals.
Dominated convergence in the identities for $f_N$ proves the assertion for $f$.
Finally,
\[
\gen f=(\Phi(s)+si)\e^{-si}h(q)+q\e^{-si}\{sh(q)-h'(q)\},
\qquad
f_i(0,q)+f_q(0,q)=h'(q)-sh(q).
\]
Substitution into \cref{eq:stationary-ito-boundary} gives \cref{eq:stationary-product-test}. \hfill
\end{proof}

\section{Further properties of the invariant law}\label{app:further-invariant-law}

\subsection{Conditional Laplace transform of idleness}\label{app:conditional-idleness-transform}

Recall that $g_Q$ is the stationary queue density and $\Xi(u)=\kappa_\alpha u^\alpha/\alpha-\beta u$.
For fixed $s>0$, we invert the joint-transform formula \cref{eq:joint-transform-reconstruction} in $r$ to identify the density of the weighted measure $\E_\pi[\e^{-sI};Q\in\dd q]$.
Dividing by the positive queue density gives the conditional Laplace transform of idleness.

\begin{corollary}[Conditional Laplace transform of idleness]\label[corollary]{cor:conditional-idleness-transform}
For Lebesgue-almost every $q>0$, simultaneously for all $s>0$,
\begin{equation}\label{eq:conditional-idleness-transform}
\E_\pi[\e^{-sI}\mid Q=q]=1-\frac{\e^{\Xi(s)}}{g_Q(q)}\int_0^q\frac{s}{v}\,\Xi'\!\left(\frac{sq}{v}\right)\e^{-\Xi(sq/v)-sq\log(q/v)}g_Q(v)\dd v.
\end{equation}
\end{corollary}

\begin{proof}
Fix a regular conditional law of $I$ given $Q$.
For $s>0$, the finite measure $\E_\pi[\e^{-sI};Q\in\dd q]$ has density $g_Q(q)\E_\pi[\e^{-sI}\mid Q=q]$ and Laplace transform $F(s,r)$, $r\ge0$.
Let $T_s(q)$ denote $\e^{\Xi(s)}$ times the integral in \cref{eq:conditional-idleness-transform}.
We show that this density equals $g_Q(q)-T_s(q)$.
Recall that $\rho_{s,r}(u)=u(r/s+\log(u/s))$ for $u\ge s>0$ and $r\ge0$.
For every $r\ge0$, integration by parts in \cref{eq:joint-transform-reconstruction} gives
\[
F(s,r)=G(r)-\e^{\Xi(s)}\int_s^\infty
\Xi'(u)\e^{-\Xi(u)}G(\rho_{s,r}(u))\dd u.
\]
Write $G(\rho_{s,r}(u))=\int_0^\infty\e^{-v\rho_{s,r}(u)}g_Q(v)\dd v$.
The resulting double integral is absolutely convergent because $\rho_{s,r}(u)\ge0$ and $\int_s^\infty|\Xi'(u)|\e^{-\Xi(u)}\dd u<\infty$, so Fubini's theorem allows us to integrate first in $u$.
At fixed $v>0$, set $u=sq/v$, so $\dd u=(s/v)\dd q$ and $u\ge s$ becomes $q\ge v$.
Since $v\rho_{s,r}(sq/v)=rq+sq\log(q/v)$, changing the order of integration over $0<v\le q$ yields
\[
\begin{aligned}
F(s,r)
&=G(r)-\e^{\Xi(s)}\int_0^\infty\e^{-rq}
\int_0^q\frac{s}{v}\,\Xi'\!\left(\frac{sq}{v}\right)
\e^{-\Xi(sq/v)-sq\log(q/v)}g_Q(v)\dd v\dd q\\
&=\int_0^\infty\e^{-rq}[g_Q(q)-T_s(q)]\dd q.
\end{aligned}
\]
The triangle inequality, Tonelli's theorem, and the substitution $u=sq/v$ give
\[
\int_0^\infty |T_s(q)|\dd q
\le \e^{\Xi(s)}\int_s^\infty
|\Xi'(u)|\e^{-\Xi(u)}G(\rho_{s,0}(u))\dd u \le \e^{\Xi(s)}\int_s^\infty
|\Xi'(u)|\e^{-\Xi(u)}\dd u<\infty,
\]
where $0\le G(\rho_{s,0}(u))\le1$ because $\rho_{s,0}(u)=u\log(u/s)\ge0$ for $u\ge s$.
Since $\int_0^\infty g_Q(q)\dd q=1$, the function $g_Q-T_s$ is integrable.
Thus $g_Q(q)\E_\pi[\e^{-sI}\mid Q=q]$ and $g_Q(q)-T_s(q)$ have the same Laplace transform for all $r\ge0$.
Uniqueness of Laplace transforms gives equality of these densities for Lebesgue-almost every $q>0$.
By \Cref{prop:joint-stationary-density}, $g_Q(q)>0$ almost everywhere, so division by $g_Q(q)$ proves \cref{eq:conditional-idleness-transform} for this fixed $s$.

We have proven the identity for each fixed $s>0$, with an exceptional null set of $q$-values that may depend on $s$.
To obtain the simultaneous assertion in the corollary, take the union of these exceptional sets over positive rational $s$, together with the null set where $g_Q$ is not finite and positive.
For fixed $q>0$, dominated convergence makes $T_s(q)$ continuous in $s$ on $(0,\infty)$.
Indeed, for $s\in[a,b]\subset(0,\infty)$, the absolute integrand defining $T_s(q)$ is bounded by $C_{a,b,q}g_Q(v)$, since $u=sq/v\ge a$ and $\sup_{u\ge a}u|\Xi'(u)|\e^{-\Xi(u)}<\infty$.
The conditional Laplace transform is also continuous in $s$, so the identity extends to every $s>0$ outside this common null set. \hfill
\end{proof}

\subsection{Joint moments and boundary singularities}\label{app:moment-refinements}

Multiplying $W_{p,\theta}(i,q)=(1+i)^p\e^{\theta q}$ by $1+i^{-a}$ introduces a boundary singularity when $a>0$.
The weighted small-idleness asymptotic \cref{eq:weighted-boundary-probability}, with $h(q)=\e^{\theta q}$, determines the resulting joint moment domain.
\begin{corollary}[Joint moment domain with a boundary singularity]\label[corollary]{cor:singular-moment-domain}
For all nonnegative $p,a,\theta$,
\begin{equation*}
\E_\pi[(1+I)^p(1+I^{-a})\e^{\theta Q}]<\infty
\quad\Longleftrightarrow\quad
p<\alpha,\quad a<\alpha-1,\quad\theta<\theta_c.
\end{equation*}
\end{corollary}
\begin{proof}
Fix $\theta<\theta_c$.
By \cref{eq:weighted-boundary-probability}, $\E_\pi[\e^{\theta Q};I\le x]$ is asymptotic to a positive finite constant times $x^{\alpha-1}$ as $x\downarrow0$.
For $a>0$, Tonelli's theorem gives
\[
\E_\pi[I^{-a}\e^{\theta Q};I\le1]
=\E_\pi[\e^{\theta Q};I\le1]+a\int_0^1x^{-a-1}\E_\pi[\e^{\theta Q};I\le x]\dd x.
\]
The small-idleness asymptotic makes this finite exactly when $a<\alpha-1$.
On $I\le1$, $(1+I)^p$ is bounded; on $I>1$, $1+I^{-a}\le2$.
\Cref{thm:invariant-moments} therefore proves sufficiency, including $a=0$.
Necessity follows because the weight dominates the marginal factors $(1+I)^p$, $I^{-a}$, and $\e^{\theta Q}$.
The corresponding expectations are infinite at their critical parameter values by the two tail limits and \Cref{thm:invariant-moments}. \hfill
\end{proof}

\subsection{Critical queue moments and transform asymptotics}\label{app:critical-queue-moments}

\begin{corollary}[Queue integrability at the critical exponential rate]\label[corollary]{cor:critical-queue-weights}
For every $r\in\R$,
\[
\begin{aligned}
\E_\pi[\e^{\theta_cQ}(1+Q)^{-r}]<\infty\quad\Longleftrightarrow\quad r>1+\alpha/2,\\
\E_\pi[I\e^{\theta_cQ}(1+Q)^{-r}]<\infty\quad\Longleftrightarrow\quad r>\alpha/2.
\end{aligned}
\]
\end{corollary}

\begin{proof}
Fix $r\in\R$ and write $h(q)=\e^{\theta_cq}(1+q)^{-r}$.
Since
\[
h'(q)=h(q)\left(\theta_c-\frac{r}{1+q}\right)
\sim\theta_c\e^{\theta_cq}q^{-r},
\]
we can choose $q_0>0$ so that $h'(q)>0$ for every $q\ge q_0$.
For a finite positive measure $\mu$ on $\R_+$, tail integration on $(q_0,\infty)$ gives
\[
\int_{(q_0,\infty)}h(q)\mu(\dd q)
=h(q_0)\mu((q_0,\infty))
+\int_{q_0}^\infty h'(q)\mu((q,\infty))\dd q.
\]
First take $\mu=\pi_Q$.
By \cref{eq:queue-sharp-tail}, the last integrand is asymptotic to a positive constant times $q^{\alpha/2-r}$, whose integral is finite exactly when $r>1+\alpha/2$.
Next take $\mu(\dd q)=\E_\pi[I;Q\in\dd q]$, which is a finite measure because $\E_\pi I=\beta$.
By \cref{eq:idleness-weighted-queue-tail}, the corresponding integrand is asymptotic to a positive constant times $q^{\alpha/2-1-r}$, whose integral is finite exactly when $r>\alpha/2$.
At equality in either criterion, the power is $-1$, so the integral diverges.
Finally, $h$ is bounded on $[0,q_0]$ and both measures are finite, so the integrals over this bounded interval are finite.
This proves both assertions. \hfill
\end{proof}

Tail integration gives the divergence rates of the exponential moments of $Q$ and $Q-I$ as $\theta\uparrow\theta_c$.
\begin{corollary}[Near-critical exponential moments]\label[corollary]{cor:critical-queue-transform}
As $\theta\uparrow\theta_c$,
\begin{equation}\label{eq:critical-queue-transform}
\E_\pi\e^{\theta Q}\sim\E_\pi\e^{\theta(Q-I)}
\sim C_\pi\theta_c\Gamma(1+\alpha/2)
(\theta_c-\theta)^{-1-\alpha/2}.
\end{equation}
\end{corollary}
\begin{proof}
For $Q\ge0$, $\E_\pi\e^{\theta Q}=1+\theta\int_0^\infty\e^{\theta x}\Pp_\pi(Q>x)\dd x$.
Use \cref{eq:queue-sharp-tail}, split off a fixed bounded interval, and set $u=(\theta_c-\theta)x$ on its complement.
Upper and lower tail comparisons give the gamma integral in \cref{eq:critical-queue-transform}.
The same argument applies to the right tail of $Q-I$; its nonpositive part contributes at most one to $\E_\pi\e^{\theta(Q-I)}$. \hfill
\end{proof}

These asymptotics also determine the exponential moments weighted by $I$ and by $I\e^{-\theta I}$.
Combining \cref{eq:idleness-weighting,eq:critical-queue-transform} gives
\[
\E_\pi[I\e^{\theta Q}]
=\frac\beta\theta\int_0^\theta\E_\pi\e^{uQ}\dd u
\sim\beta C_\pi\Gamma(\alpha/2)(\theta_c-\theta)^{-\alpha/2} \qquad(\theta\uparrow\theta_c).
\]
For $0<\theta<\theta_c$, take $s=\theta$ and $h(q)=\e^{\theta q}$ in \cref{eq:stationary-product-test}.
Since $sh-h'=0$, the regulator term vanishes and the identity reduces to
\begin{equation}\label{eq:diagonal-transform-identity}
\E_\pi[I\e^{\theta(Q-I)}]
=(\beta-\kappa_\alpha\theta^{\alpha-1})\E_\pi[\e^{\theta(Q-I)}].
\end{equation}
Combining \cref{eq:diagonal-transform-identity,eq:critical-queue-transform} with $\beta-\kappa_\alpha\theta^{\alpha-1}\sim\mu_c(\theta_c-\theta)/\theta_c$ gives
\[
\E_\pi[I\e^{\theta(Q-I)}]
\sim\mu_c C_\pi\Gamma(1+\alpha/2)(\theta_c-\theta)^{-\alpha/2} \qquad(\theta\uparrow\theta_c).
\]
Consequently,
\begin{equation*}
\frac{\E_\pi[I\e^{\theta(Q-I)}]}{\E_\pi[I\e^{\theta Q}]}
\longrightarrow\frac{\mu_c\alpha}{2\beta}
=\frac{\alpha(\alpha-1)}2
\qquad(\theta\uparrow\theta_c).
\end{equation*}
Thus $\e^{-\theta I}$ changes the leading constant of the idleness-weighted exponential moments.

\clearpage
\bibliographystyle{plainnat}
\bibliography{references}

\end{document}